\documentclass{amsart}
\usepackage{amsmath, amssymb, amsthm, mathrsfs, amscd}
\usepackage{graphicx}

\theoremstyle{plain} 
 \newtheorem{thm}{Theorem}[section]
 \newtheorem{lem}[thm]{Lemma}
 \newtheorem{cor}[thm]{Corollary}
 \newtheorem{prop}[thm]{Proposition}
 \newtheorem{claim}[thm]{Claim}

\theoremstyle{definition}
  \newtheorem{defn}[thm]{Definition}

\theoremstyle{remark}

\renewcommand{\mod}{{\rm Mod}}
\renewcommand{\pmod}{{\rm PMod}}
\newcommand{\cal}{\mathcal}
\newcommand{\calc}{\mathcal{C}}
\newcommand{\calcp}{\mathcal{CP}}
\newcommand{\cald}{\mathcal{D}}
\newcommand{\cala}{\mathcal{A}}
\newcommand{\lk}{{\rm Lk}}
\newcommand{\ci}[2]{\cite[#1]{#2}}

\begin{document}

\title{The co-Hopfian property of surface braid groups}
\author{Yoshikata Kida}
\author{Saeko Yamagata}
\address{Department of Mathematics, Kyoto University, 606-8502 Kyoto, Japan}
\email{kida@math.kyoto-u.ac.jp}
\address{Faculty of Education and Human Sciences, Yokohama National University, 240-8501 Yokohama, Japan}
\email{yamagata@ynu.ac.jp}
\date{March 1, 2013}

\begin{abstract}
Let $g$ and $n$ be integers at least 2, and let $G$ be the pure braid group with $n$ strands on a closed orientable surface of genus $g$.
We describe any injective homomorphism from a finite index subgroup of $G$ into $G$. 
As a consequence, we show that any finite index subgroup of $G$ is co-Hopfian.
\end{abstract}

\maketitle


\section{Introduction}

Let $S$ be a connected, compact and orientable surface which may have non-empty boundary.
Let $\mod^*(S)$ be the extended mapping class group of $S$.
A description of any isomorphism between two finite index subgroups of $\mod^*(S)$ was obtained in \cite{iva-aut}, \cite{kork-aut} and \cite{luo}.
A key step in these works is to compute the automorphism group of the complex of curves for $S$, denoted by $\calc(S)$.
More generally, to describe any injective homomorphism from a finite index subgroup of $\mod^*(S)$ into $\mod^*(S)$, superinjectivity of a simplicial map from $\calc(S)$ into itself was introduced by Irmak \cite{irmak1}.
Any such superinjective map is shown to be induced by an element of $\mod^*(S)$ in \cite{be-m}, \cite{bm-ar}, \cite{irmak1}, \cite{irmak2} and \cite{irmak-ns}.
The same conclusion was obtained for injective simplicial maps from $\calc(S)$ into itself by Shackleton \cite{sha}.
To prove similar results on the Torelli group and the Johnson kernel for a certain surface, variants of the complex of curves are introduced and studied in \cite{bm}, \cite{bm-add}, \cite{farb-ivanov}, \cite{kida-tor}, \cite{kida-cohop}, \cite{kida-inj} and \cite{mv}.

Let $\bar{S}$ be the closed surface obtained from $S$ by attaching disks to all boundary components of $S$.
We then have the homomorphism from the pure mapping class group $\pmod(S)$ of $S$ onto the mapping class group $\mod(\bar{S})$ of $\bar{S}$.
We define $P(S)$ as the kernel of this homomorphism. 
If the genus of $S$ is at least 2, then $P(S)$ is naturally identified with the fundamental group of the space of ordered distinct $p$ points in $\bar{S}$, where $p$ denotes the number of boundary components of $S$, as discussed in \cite[Theorem 4.2]{birman}.
In our previous paper \cite{kida-yama}, we showed that for some surfaces $S$, any isomorphism between two finite index subgroups of $P(S)$ is the conjugation by an element of $\mod^*(S)$.
The purpose of the present paper is to establish the same conclusion for any injective homomorphism from a finite index subgroup of $P(S)$ into $P(S)$.
As its immediate consequence, we show that any finite index subgroup $\Gamma$ of $P(S)$ is {\it co-Hopfian}, that is, any injective homomorphism from $\Gamma$ into itself is surjective.

In \cite{kida-yama}, inspired by a work due to Irmak-Ivanov-McCarthy \cite{iim}, we introduced the simplicial complex $\calcp(S)$ called the complex of HBCs and HBPs in $S$, on which $\mod^*(S)$ naturally acts.
We proved that the automorphism group of $\calcp(S)$ is naturally isomorphic to $\mod^*(S)$.
We also proved that any injective homomorphism from a finite index subgroup of $P(S)$ into $P(S)$ induces a superinjective map from $\calcp(S)$ into itself.
The present paper is thus devoted to showing that any superinjective map from $\calcp(S)$ into itself is surjective.
The following theorem is a consequence of these results.

\begin{thm}\label{thm-b}
Let $S$ be a connected, compact and orientable surface such that both the genus and the number of boundary components of $S$ are at least 2.
Then any superinjective map from $\calcp(S)$ into itself is induced by an element of $\mod^*(S)$.
\end{thm} 

As a by-product of the proof of this theorem, in Corollary \ref{cor-bn}, we also prove that any superinjective map from the subcomplex $\calcp_n(S)$ of $\calcp(S)$, called the complex of HBCs and non-separating HBPs for $S$, into $\calcp(S)$ is induced by an element of $\mod^*(S)$.
Combining the above theorem with \cite[Theorem 7.13 (i)]{kida-yama} and applying argument in \cite[Section 3]{iva-aut}, we obtain the following:

\begin{cor}\label{cor-cohop}
Let $S$ be the surface in Theorem \ref{thm-b}.
Then for any finite index subgroup $\Gamma$ of $P(S)$ and any injective homomorphism $f\colon \Gamma \to P(S)$, there exists a unique element $\gamma \in \mod^*(S)$ with $f(x)=\gamma x\gamma^{-1}$ for any $x\in \Gamma$.
In particular, $\Gamma$ is co-Hopfian.
\end{cor}

Parallel results are proved for the subgroup $P_s(S)$ of $P(S)$ and the subcomplex $\calcp_s(S)$ of $\calcp(S)$, called the complex of HBCs and separating HBPs for $S$, that are precisely defined in Section \ref{subsec-cpx}.
Let us summarize the results on them.

\begin{thm}\label{thm-bs}
Let $S$ be the surface in Theorem \ref{thm-b}.
Then any superinjective map from $\calcp_s(S)$ into itself is induced by an element of $\mod^*(S)$.
\end{thm}

\begin{cor}
Let $S$ be the surface in Theorem \ref{thm-b}.
Then for any finite index subgroup $\Lambda$ of $P_s(S)$ and any injective homomorphism $h\colon \Lambda \to P_s(S)$, there exists a unique element $\lambda \in \mod^*(S)$ with $h(y)=\lambda y\lambda^{-1}$ for any $y\in \Lambda$.
In particular, $\Lambda$ is co-Hopfian.
\end{cor}

This corollary is obtained by combining the last theorem with \cite[Theorem 7.13 (ii)]{kida-yama}.
If the genus of a surface $S$ is equal to 0, then we have the equalities $P(S)=P_s(S)=\pmod(S)$ and $\calcp(S)=\calcp_s(S)=\calc(S)$.
If the genus of $S$ is equal to 1, then we have the equalities $P(S)=\mathcal{I}(S)$ and $P_s(S)=\mathcal{K}(S)$, where $\mathcal{I}(S)$ is the Torelli group for $S$ and $\mathcal{K}(S)$ is the Johnson kernel for $S$.
Moreover, $\calcp(S)$ and $\calcp_s(S)$ are equal to the Torelli complex for $S$ and the complex of separating curves for $S$, respectively.
We refer to \cite{kida-tor} for a definition of these groups and complexes.
It therefore follows from \cite{bm-ar} and \cite{kida-cohop} that if $S$ is a surface with the genus less than 2 and the Euler characteristic less than $-2$, then the same conclusions for $S$ as in the above theorems hold.
The co-Hopfian property of the braid groups on the disk, which are central extensions of the mapping class groups of holed spheres, is discussed in \cite{bm-br} and \cite{bellingeri}.

For a positive integer $n$ and a manifold $M$, we define $B_n(M)$ as the {\it braid group} of $n$ strands on $M$, i.e., the fundamental group of the space of non-ordered distinct $n$ points in $M$.
We also define $PB_n(M)$ as the {\it pure braid group} of $n$ strands on $M$, i.e., the fundamental group of the space of ordered distinct $n$ points in $M$.
The group $PB_n(M)$ is naturally identified with a subgroup of $B_n(M)$ of index $n!$.
We refer to \cite{birman} and \cite{pr} for fundamental facts on these groups.

Let $S$ be a surface of genus at least 2 with $p$ boundary components.
The kernel of the homomorphism from $\mod(S)$ onto $\mod(\bar{S})$ associated with the inclusion of $S$ into $\bar{S}$ is then identified with $B_p(\bar{S})$ (see \cite[Theorem 4.3]{birman}).
Under this identification, $P(S)$ is identified with $PB_p(\bar{S})$.
As a consequence of Corollary \ref{cor-cohop}, we obtain the following corollary, which answers to \cite[Question 4]{bm-br} affirmatively for closed orientable surfaces of genus at least 2.

\begin{cor}\label{cor-braid}
Let $n$ be an integer at least 2.
Let $M$ be a connected, closed and orientable surface of genus at least 2.
Then any finite index subgroup of $B_n(M)$ is co-Hopfian.
\end{cor}

The proof of this corollary is presented in Section \ref{sec-cor}.
In the next section, we introduce notation and terminology employed throughout the paper, and recall basic properties of superinjective maps defined on $\calcp(S)$ etc. 
Afterward, we present an organization of the paper, outlining the proof of Theorems \ref{thm-b} and \ref{thm-bs}.


\section{Preliminaries}\label{sec-pre}

\subsection{Notation and terminology}\label{subsec-term}

Unless otherwise stated, we assume that a surface is connected, compact and orientable. 
Let $S=S_{g, p}$ be a surface of genus $g$ with $p$ boundary components, and let $\partial S$ denote the boundary of $S$. 
A simple closed curve in $S$ is called {\it essential} in $S$ if it is neither homotopic to a single point of $S$ nor isotopic to a boundary component of $S$. 
Let $V(S)$ denote the set of isotopy classes of essential simple closed curves in $S$. 
When there is no confusion, we mean by a curve in $S$ either an essential simple closed curve in $S$ or its isotopy class. 
An essential simple closed curve $\alpha$ in $S$ is called {\it non-separating} in $S$ if $S\setminus \alpha$ is connected, and otherwise $\alpha$ is called {\it separating} in $S$.
These properties depend only on the isotopy class of $\alpha$.
For a separating curve $\alpha$ in $S$ and two components $\partial_1$, $\partial_2$ of $\partial S$, we say that $\alpha$ {\it separates} $\partial_1$ and $\partial_2$ if $\partial_1$ and $\partial_2$ are contained in distinct components of $S\setminus \alpha$.

Let $i\colon V(S)\times V(S)\to \mathbb{Z}_{\geq 0}$ denote the {\it geometric intersection number}, i.e., the minimal cardinality of the intersection of representatives for two elements of $V(S)$.
If $a$ and $b$ are essential simple closed curves in $S$ with $|a\cap b|=i([a], [b])$, where $[a]$ and $[b]$ denote the isotopy classes of $a$ and $b$, respectively, then we say that $a$ and $b$ {\it intersect minimally}.
Let $\Sigma(S)$ denote the set of non-empty finite subsets $\sigma$ of $V(S)$ with $i(\alpha, \beta)=0$ for any two elements $\alpha, \beta \in \sigma$.
We extend $i$ to the symmetric function on $(V(S)\sqcup \Sigma(S))^2$ so that $i(\alpha, \sigma)=\sum_{\beta \in \sigma}i(\alpha, \beta)$ and $i(\sigma, \tau)=\sum_{\beta \in \sigma, \gamma \in \tau}i(\beta, \gamma)$ for any $\alpha \in V(S)$ and $\sigma, \tau \in \Sigma(S)$.
We say that two elements $\sigma$, $\tau$ of $V(S)\sqcup \Sigma(S)$ are {\it disjoint} if $i(\sigma, \tau)=0$, and otherwise we say that they {\it intersect}.
We say that two elements $\alpha$, $\beta$ of $V(S)$ {\it fill} $S$ if there exists no element of $V(S)$ disjoint from both $\alpha$ and $\beta$.

For $\sigma \in \Sigma(S)$, we denote by $S_{\sigma}$ the surface obtained by cutting $S$ along all curves in $\sigma$.
When $\sigma$ consists of a single curve $\alpha$, we denote it by $S_{\alpha}$ for simplicity.
We often identify a component of $S_{\sigma}$ with a complementary component of a tubular neighborhood of a one-dimensional submanifold representing $\sigma$ in $S$ if there is no confusion.
If $Q$ is a component of $S_{\sigma}$, then $V(Q)$ is naturally identified with a subset of $V(S)$.

Suppose that the boundary $\partial S$ of $S$ is non-empty.
We say that a simple arc $l$ in $S$ is {\it essential} in $S$ if
\begin{itemize}
\item $\partial l$ consists of two distinct points of $\partial S$;
\item $l$ meets $\partial S$ only at its end points; and
\item $l$ is not isotopic relative to $\partial l$ to an arc in $\partial S$.
\end{itemize}
Unless otherwise stated, isotopy of essential simple arcs in $S$ may move their end points, keeping them staying in $\partial S$.
An essential simple arc $l$ in $S$ is called {\it separating} in $S$ if $S\setminus l$ is not connected.
Otherwise, $l$ is called {\it non-separating} in $S$.
These properties depend only on the isotopy class of $l$.
Let $\partial_1$ and $\partial_2$ be distinct components of $\partial S$.
We say that an essential simple arc $l$ in $S$ {\it connects} $\partial_1$ {\it and} $\partial_2$ (or {\it connects} $\partial_1$ {\it with} $\partial_2$) if one of the end point of $l$ lies in $\partial_1$ and another in $\partial_2$.

\subsection{Complexes and groups associated to surfaces}\label{subsec-cpx}

The following simplicial complex was introduced by Harvey \cite{harvey}.

\medskip

\noindent {\bf Complex $\calc(S)$.} This is defined as the abstract simplicial complex such that the sets of vertices and simplices of $\calc(S)$ are $V(S)$ and $\Sigma(S)$, respectively.
The complex $\calc(S)$ is called the {\it complex of curves} for $S$.

\medskip

We denote by $\bar{S}$ the closed surface obtained from $S$ by attaching disks to all boundary components of $S$. 
Let $\calc^*(\bar{S})$ be the simplicial cone over $\calc(\bar{S})$ with its cone point $\ast$.
Namely, $\calc^*(\bar{S})$ is the abstract simplicial complex such that the set of vertices is the disjoint union $V(\bar{S})\sqcup \{ \ast \}$, and the set of simplices is
\[\Sigma(\bar{S})\cup \{\, \sigma \cup \{ \ast \} \mid \sigma \in \Sigma(\bar{S})\cup \{ \emptyset \}\,\}.\]
We then have the simplicial map
\[\pi \colon \calc(S)\to \calc^*(\bar{S})\]
associated with the inclusion of $S$ into $\bar{S}$.
Note that $\pi^{-1}(\ast)$ consists of all separating curves in $S$ cutting off a holed sphere from $S$.

\medskip

\noindent {\bf Hole-bounding curves (HBC).} A curve $\alpha$ in $S$ is called a {\it hole-bounding curve (HBC)} in $S$ if $\alpha$ lies in $\pi^{-1}(\ast)$. 
When the genus of $S$ is positive and the holed sphere cut off by $\alpha$ contains exactly $k$ components of $\partial S$, we call $\alpha$ a {\it $k$-HBC} in $S$.
\begin{figure}
\begin{center}
\includegraphics[width=8cm]{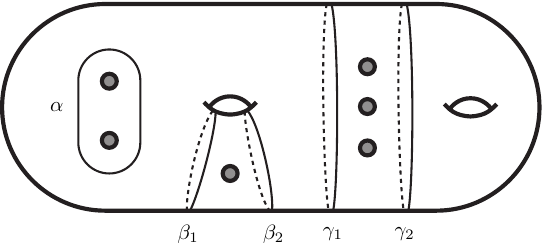}
\caption{The curve $\alpha$ is a 2-HBC, the pair $\{ \beta_1, \beta_2\}$ is a non-separating 1-HBP, and the pair $\{ \gamma_1, \gamma_2\}$ is a separating 3-HBP.}\label{fig-exhbchbp}
\end{center}
\end{figure}
We have $2\leq k\leq p$ (see Figure \ref{fig-exhbchbp}).

\medskip

If $\alpha$ is a $k$-HBC in $S$ and $\partial_1,\ldots, \partial_k$ are the components of $\partial S$ contained in the holed sphere cut off by $\alpha$, then we say that $\alpha$ {\it encircles} $\partial_1,\ldots, \partial_k$.

Let $\alpha$ be a 2-HBC in $S$ encircling two components $\partial_1$, $\partial_2$ of $\partial S$.
Up to isotopy, there exists a unique essential simple arc in $S$ connecting $\partial_1$ with $\partial_2$ and disjoint from $\alpha$.
This arc is called the {\it defining arc} of $\alpha$.
Conversely, if we have $g\geq 1$ and $p\geq 2$ and if $l$ is an essential simple arc in $S$ connecting two distinct components $\partial_1$, $\partial_2$ of $\partial S$, then the boundary component of a regular neighborhood of the union $l\cup \partial_1\cup \partial_2$ is a 2-HBC in $S$.
The 2-HBC is then called the curve in $S$ {\it defined by} $l$.

\medskip

\noindent {\bf Hole-bounding pairs (HBP).} A pair $\{ \alpha, \beta \}$ of curves in $S$ is called a {\it hole-bounding pair (HBP)} in $S$ if $\{ \alpha, \beta \}$ is an edge of $\calc(S)$ and we have $\pi(\alpha)=\pi(\beta)\neq \ast$.
We note that there exists a unique component of $S_{\{ \alpha, \beta\}}$ of genus 0 if the genus of $S$ is at least 2.
In this case, if that component contains exactly $k$ components of $\partial S$, then we call the pair $\{ \alpha, \beta \}$ a {\it $k$-HBP} in $S$.
We have $1\leq k\leq p$ (see Figure \ref{fig-exhbchbp}).

An HBP in $S$ is called {\it non-separating} if both curves in it are non-separating in $S$, and it is called {\it separating} if both curves in it are separating in $S$.
Any HBP in $S$ is either separating or non-separating.
Two disjoint HBPs $a$, $b$ in $S$ are called {\it equivalent} in $S$ if $\pi(a)=\pi(b)$.
Two disjoint curves $\alpha$, $\beta$ in $S$ are called {\it HBP-equivalent} in $S$ if $\pi(\alpha)=\pi(\beta)$.

\medskip

\noindent {\bf Complexes $\calcp(S)$, $\calcp_n(S)$ and $\calcp_s(S)$.} Let $V_c(S)$ be the set of all isotopy classes of HBCs in $S$, and let $V_p(S)$ be the set of all isotopy classes of HBPs in $S$. 
We define $\calcp(S)$ as the abstract simplicial complex such that the set of vertices is the disjoint union $V_c(S)\sqcup V_p(S)$, and a non-empty finite subset $\sigma$ of $V_c(S)\sqcup V_p(S)$ is a simplex of $\calcp(S)$ if and only if any two elements of $\sigma$ are disjoint.

We denote by $V_{np}(S)$ and $V_{sp}(S)$ the subsets of $V_p(S)$ consisting of all non-separating HBPs and all separating HBPs, respectively.
We define $\calcp_n(S)$ and $\calcp_s(S)$ as the full subcomplexes of $\calcp(S)$ spanned by $V_c(S)\sqcup V_{np}(S)$ and $V_c(S)\sqcup V_{sp}(S)$, respectively.

\medskip

\noindent {\bf Mapping class groups.} The {\it extended mapping class group} $\mod^*(S)$ of $S$ is the group of isotopy classes of homeomorphisms from $S$ onto itself, where isotopy may move points in $\partial S$. 
The {\it mapping class group} $\mod(S)$ of $S$ is the group of isotopy classes of orientation-preserving homeomorphisms from $S$ onto itself.
The {\it pure mapping class group} $\pmod(S)$ of $S$ is the group of isotopy classes of orientation-preserving homeomorphisms from $S$ onto itself that fix each component of $\partial S$ as a set.
Both $\mod(S)$ and $\pmod(S)$ are finite index subgroups of $\mod^*(S)$.
For $\alpha \in V(S)$, we denote by $t_{\alpha}\in \pmod(S)$ the {\it (left) Dehn twist} about $\alpha$.

\medskip

\noindent {\bf Surface braid groups.} We have the surjective homomorphism
\[\iota \colon \pmod(S)\to \mod(\bar{S})\]
associated with the inclusion of $S$ into $\bar{S}$.
We define $P(S)$ to be $\ker \iota$. 
By the Birman exact sequence, $P(S)$ is generated by all elements of the forms $t_{\alpha}$ and $t_{\beta}t_{\gamma}^{-1}$ with $\alpha \in V_c(S)$ and $\{ \beta, \gamma \}\in V_p(S)$ (see \cite[Section 4.1]{birman}). 
We define $P_s(S)$ as the group generated by all elements of the forms $t_{\alpha}$ and $t_{\beta}t_{\gamma}^{-1}$ with $\alpha \in V_c(S)$ and $\{ \beta, \gamma \} \in V_{sp}(S)$.
The group $P_s(S)$ is of infinite index in $P(S)$ because no non-zero power of $t_{\delta}t_{\epsilon}^{-1}$ with $\{ \delta, \epsilon \}\in V_{np}(S)$ belongs to $P_s(S)$ by \cite[Lemma 2.3]{kida-yama}.


\subsection{Superinjective maps}\label{subsec-si}

We review basic properties of superinjective maps defined between simplicial complexes introduced in the previous subsection.

\begin{defn}
Let $S$ be a surface.
Let $X$ and $Y$ be any of the simplicial complexes, $\calc(S)$, $\calcp(S)$, $\calcp_n(S)$ and $\calcp_s(S)$.  
We mean by a {\it superinjective map} $\phi \colon X\to Y$ a simplicial map $\phi \colon X\to Y$ satisfying $i(\phi(a), \phi(b))\neq 0$ for any two vertices $a$, $b$ of $X$ with $i(a, b)\neq 0$.
\end{defn}

Let $V(X)$ and $V(Y)$ denote the sets of vertices of $X$ and $Y$, respectively.
We note that a map $\phi \colon V(X)\to V(Y)$ defines a simplicial map from $X$ into $Y$ if and only if $i(\phi(a), \phi(b))=0$ for any two vertices $a, b\in V(X)$ with $i(a, b)=0$.
It can be checked that any superinjective map from $X$ into $Y$ is injective, and that for any superinjective map $\phi \colon X\to Y$, if the induced map from $V(X)$ into $V(Y)$ is surjective, then $\phi$ is a simplicial isomorphism from $X$ onto $Y$.

\begin{lem}\label{lem-hbp}
Let $S=S_{g, p}$ be a surface with $g\geq 2$ and $p\geq 2$.
Let $\phi \colon \calcp_n(S)\to \calcp(S)$ and $\psi \colon \calcp_s(S)\to \calcp(S)$ be superinjective maps.
Then $\phi$ sends each non-separating HBP to an HBP, and $\psi$ sends each separating HBP to an HBP.
\end{lem}

This lemma is a direct consequence of the following proposition describing simplices of $\calcp(S)$ of maximal dimension (see Figure \ref{fig-max} for such simplices).
\begin{figure}
\begin{center}
\includegraphics[width=12cm]{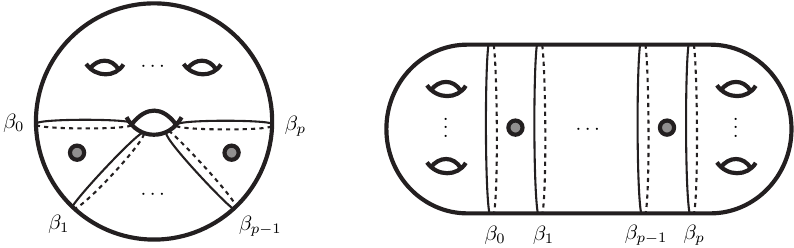}
\caption{Simplices of $\calcp(S)$ of maximal dimension}\label{fig-max}
\end{center}
\end{figure}

\begin{prop}[\ci{Proposition 3.5}{kida-yama}]\label{prop-max}
Let $S=S_{g, p}$ be a surface with $g\geq 2$ and $p\geq 2$.
Then we have
\[\dim(\calcp(S))=\dim(\calcp_s(S))=\dim(\calcp_n(S))=\binom{p+1}{2}-1.\]
Moreover, for any simplex $\sigma$ of $\calcp(S)$ of maximal dimension, there exists a unique simplex $s=\{ \beta_0, \beta_1,\ldots, \beta_p\}$ of $\calc(S)$ such that
\begin{itemize}
\item any two curves in $s$ are HBP-equivalent; and
\item $\sigma$ consists of all HBPs of two curves in $s$.
\end{itemize}
\end{prop}

Using this description of simplices of maximal dimension, we can also show that the maps $\phi$ and $\psi$ in Lemma \ref{lem-hbp} preserve rooted simplices defined as follows.

\begin{defn}
Let $S$ be a surface, and let $\sigma$ be a simplex of $\calcp(S)$ consisting of HBPs.
We say that $\sigma$ is {\it rooted} if there exists a curve $\alpha$ in $S$ contained in any HBP of $\sigma$.
In this case, if $|\sigma|\geq 2$, then $\alpha$ is uniquely determined and called the {\it root curve} of $\sigma$.
\end{defn}

Rooted simplices were first introduced in \cite{kida-tor} for the Torelli complex of $S$ in an analogous way. 
The proof of the following two lemmas are verbatim translations of those of the cited lemmas, where only superinjective maps from $\calcp(S)$ into itself and ones from $\calcp_s(S)$ into itself are dealt with.

\begin{lem}[\ci{Lemma 3.12}{kida-yama}]\label{lem-rooted}
Let $S=S_{g, p}$ be a surface with $g \geq 2$ and $p \geq 2$. 
Then any superinjective map from $\calcp_n(S)$ into $\calcp(S)$ preserves rooted simplices.
Moreover, the same conclusion holds for any superinjective map from $\calcp_s(S)$ into $\calcp(S)$.
\end{lem}

\begin{lem}[\ci{Lemma 3.13}{kida-yama}]\label{lem-root-curve}
Let $S=S_{g, p}$ be a surface with $g \geq 2$ and $p \geq 2$.
Let $\phi \colon \calcp_n(S) \to \calcp(S)$ be a superinjective map.
Pick a simplex $\sigma$ of $\calcp_n(S)$ of maximal dimension.
We denote by $\{ \alpha_0,\alpha_1,\ldots, \alpha_p\}$ the collection of curves in HBPs of $\sigma$. 
Then there exists a collection of curves in $S$, $\{ \beta_0,\beta_1,\ldots, \beta_p\}$, satisfying the equality
\[\phi(\{ \alpha_j, \alpha_k\})=\{ \beta_j, \beta_k\}\]
for any distinct $j, k=0,1,\ldots, p$.
Moreover, the same conclusion holds for any superinjective map from $\calcp_s(S)$ into $\calcp(S)$ and any simplex of $\calcp_s(S)$ of maximal dimension.
\end{lem}


\subsection{Outline}

Let $S=S_{g, p}$ be a surface with $g\geq 2$ and $p\geq 2$.
Our aim is to show surjectivity of any superinjective map from $\calcp(S)$ into itself and any superinjective map from $\calcp_s(S)$ into itself.
In Section \ref{sec-bn-b}, we discuss topological types of vertices preserved under a superinjective map $\phi \colon \calcp_n(S)\to \calcp(S)$.
More precisely, we show that for any integers $j$ and $k$ with $2\leq j\leq p$ and $1\leq k\leq p$, the map $\phi$ preserves $j$-HBCs and non-separating $k$-HBPs, respectively.
The inclusion $\phi(\calcp_n(S))\subset \calcp_n(S)$ therefore holds.
In Section \ref{sec-bs-bs}, we deal with a superinjective map $\psi \colon \calcp_s(S)\to \calcp_s(S)$, and similarly show that for any $j$ and $k$, the map $\psi$ preserves $j$-HBCs and separating $k$-HBPs, respectively.

In Section \ref{sec-bn}, we focus on the case of $p=2$, and discuss topological types of hexagons in $\calcp_n(S)$, provided that topological types of their vertices are given.
As a result, for any superinjective map $\phi \colon \calcp_n(S)\to \calcp_n(S)$ and any non-separating 2-HBP $a$ in $S$, the map $\phi$ sends the link of $a$ onto the link of $\phi(a)$.
In Section \ref{sec-bs}, we still focus on the case of $p=2$, and obtain a similar result for any superinjective map $\psi \colon \calcp_s(S)\to \calcp_s(S)$ and any separating 2-HBP $b$ in $S$.
Namely, we show that $\psi$ sends the link of $b$ onto the link of $\psi(b)$.

In Section \ref{sec-d}, to obtain a similar surjectivity on the link of any 1-HBP, we introduce two natural subcomplexes of the complex of arcs, denoted by $\cald(X, \partial)$ and $\cald(Y)$.
We show that any injective simplicial maps from those complexes into themselves are surjective.

In Section \ref{sec-surj}, we prove that any superinjective maps $\phi \colon \calcp_n(S)\to \calcp_n(S)$ and $\psi \colon \calcp_s(S)\to \calcp_s(S)$ are surjective.
This is proved by induction on $p$.
In the case of $p=2$, we show that for any non-separating 1-HBP $a$, the restriction of $\phi$ to the link of $a$ induces an injective simplicial map from $\cald(X, \partial)$ into itself.
It follows from the result in Section \ref{sec-d} that $\phi$ sends the link of $a$ onto the link of $\phi(a)$.
Similarly, using $\cald(Y)$ in place of $\cald(X, \partial)$, we show that for any separating 1-HBP $b$, the map $\psi$ sends the link of $b$ onto the link of $\psi(b)$.
Combining surjectivity on the link of a 2-HBP proved in Sections \ref{sec-bn} and \ref{sec-bs}, we obtain surjectivity of $\phi$ and $\psi$.

Finally, surjectivity of any $\phi$ and $\psi$ leads to surjectivity of any superinjective map from $\calcp(S)$ into itself.


\subsection{Proof of Corollary \ref{cor-braid}}\label{sec-cor}

Assuming Corollary \ref{cor-cohop}, we prove Corollary \ref{cor-braid}.
Let $S=S_{g, p}$ be a surface with $g\geq 2$ and $p\geq 2$.
We denote by $\jmath \colon \mod(S)\to \mod(\bar{S})$ the homomorphism associated with the inclusion of $S$ into $\bar{S}$, which is an extension of $\iota$.
We define $B(S)$ as the kernel of $\jmath$.
Our aim is to prove that any finite index subgroup of $B(S)$ is co-Hopfian.

Let $\Gamma$ be a finite index subgroup of $B(S)$.
Let $f\colon \Gamma \to \Gamma$ be an injective homomorphism.
We put $\Gamma_0=P(S)\cap f^{-1}(P(S)\cap \Gamma)$, which is a finite index subgroup of $P(S)$.
By Corollary \ref{cor-cohop}, there exists $\gamma \in \mod^*(S)$ with $f(x)=\gamma x\gamma^{-1}$ for any $x\in \Gamma_0$.
For any $y\in \Gamma$ and any HBC $\alpha$ in $S$, there exists a non-zero integer $N$ such that $t_{\alpha}^N$ and $t_{y\alpha}^N$ belong to $\Gamma_0$.
We then have
\[f(yt_{\alpha}^Ny^{-1})=f(y)f(t_{\alpha}^N)f(y)^{-1}=f(y)\gamma t_{\alpha}^N\gamma^{-1}f(y)^{-1}=t_{f(y)\gamma \alpha}^N.\]
On the other hand, we have
\[f(yt_{\alpha}^Ny^{-1})=f(t_{y\alpha}^N)=\gamma t_{y\alpha}^N\gamma^{-1}=t_{\gamma y\alpha}^N.\]
We thus have $f(y)\gamma \alpha =\gamma y\alpha$.
Replacing $\alpha$ with any HBP $b=\{ b_1, b_2\}$ in $S$ and replacing $t_{\alpha}$ with $t_{b_1}t_{b_2}^{-1}$, we can show the equality $f(y)\gamma b=\gamma yb$ along a verbatim argument.
Since the action of $\mod^*(S)$ on $\calcp(S)$ is faithful by \cite[Lemma 2.2]{kida-yama}, the equality $f(y)=\gamma y\gamma^{-1}$ holds for any $y\in \Gamma$.

We have $\gamma^{-1}B(S)\gamma =B(S)$ because $B(S)$ is a normal subgroup of $\mod^*(S)$.
It follows that
\[[B(S):\Gamma]\leq [B(S):f(\Gamma)]=[B(S):\gamma \Gamma \gamma^{-1}]=[B(S): \Gamma].\]
The equality $[B(S):\Gamma]=[B(S):f(\Gamma)]$ holds.
We thus have $f(\Gamma)=\Gamma$.
Corollary \ref{cor-braid} is proved.


\section{Superinjective maps from $\calcp_n(S)$ into $\calcp(S)$}\label{sec-bn-b}

We mean by a {\it pentagon} in $\calcp(S)$ a subgraph of $\calcp(S)$ consisting of five vertices $v_1,\ldots, v_5$ with $i(v_k, v_{k+1})=0$ and $i(v_k, v_{k+2})\neq 0$ for any $k$ mod $5$. 
In this case, let us say that the pentagon is defined by the 5-tuple $(v_1,\ldots, v_{5})$.
A pentagon in $\calcp_n(S)$ is defined as a pentagon in $\calcp(S)$ any of whose vertices belongs to $\calcp_n(S)$.
Examples of pentagons in $\calcp_n(S_{g, 2})$ are drawn in Figures \ref{fig-nhbp-pen} and \ref{fig-nhbp-c-pen}.
\begin{figure}
\begin{center}
\includegraphics[width=10cm]{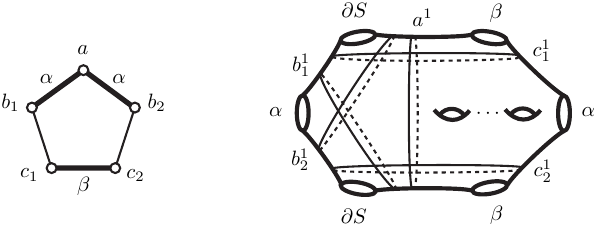}
\caption{A pentagon in $\calcp_n(S_{g, 2})$ with $a=\{ \alpha, a^1\}$, $b_j=\{ \alpha, b_j^1\}$ and $c_j=\{ \beta, c_j^1\}$ for $j=1, 2$. Thick edges of the pentagon in the left hand side and nearby symbols signify rooted edges and their root curves, respectively. Rooted edges and their root curves are similarly indicated in other Figures. The surface obtained by cutting $S_{g, 2}$ along $\alpha$ and $\beta$ is drawn in the right hand side.}\label{fig-nhbp-pen}
\end{center}
\end{figure}
The following two lemmas are observations on pentagons in $\calcp(S_{g, 2})$ consisting of HBPs.

\begin{lem}\label{lem-no-pent}
Let $S=S_{g, 2}$ be a surface with $g\geq 2$. 
Then there exists no pentagon $\Pi$ in $\calcp(S)$ such that
\begin{itemize}
\item any vertex of $\Pi$ corresponds to an HBP; and
\item there exists a curve in $S$ contained in any HBP of $\Pi$.
\end{itemize}
\end{lem}

\begin{proof}
Assuming that there exists such a $\Pi$ defined by a 5-tuple $(a, b, c, d, e)$, we deduce a contradiction. 
Let $\alpha$ denote the curve shared by the five HBPs of $\Pi$. 
Label all boundary components of $S_{\alpha}$ by $\alpha_1$, $\alpha_2$, $\partial_1$ and $\partial_2$ so that $\alpha_1$ and $\alpha_2$ correspond to $\alpha$, and $\partial_1$ and $\partial_2$ are components of $\partial S$. 
We define $a^1$ as the curve in $a$ distinct from $\alpha$, and define $b^1$, $c^1$, $d^1$ and $e^1$ similarly.

We claim that at least one of vertices of $\Pi$ is a 2-HBP.
Suppose that this is not true.
If $x$ and $y$ are distinct and disjoint 1-HBPs in $S$, then the pair of pants cut off by $x$ from $S$ and that cut off by $y$ from $S$ contain distinct components of $\partial S$.
Apply this to each edge of $\Pi$.
We then obtain a contradiction because the number of edges of $\Pi$ is odd.
The claim follows.

We may therefore assume that $a$ is a 2-HBP and that $a^1$ encircles $\alpha_1$, $\partial_1$ and $\partial_2$ as a curve in $S_{\alpha}$ if $\alpha$ is non-separating, or as a curve in a component of $S_{\alpha}$ if $\alpha$ is separating.
It follows that $b$ and $e$ are 1-HBPs. 
Without loss of generality, we may assume that $b^1$ encircles $\alpha_1$ and $\partial_1$ (see Figure \ref{fig-hbp-pf}).
\begin{figure}
\begin{center}
\includegraphics[width=11cm]{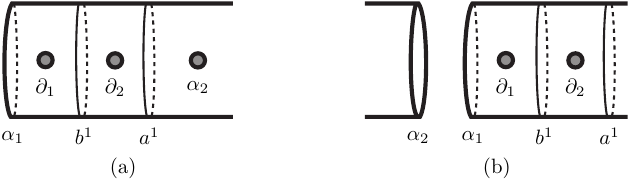}
\caption{The surface $S_{\alpha}$ is drawn.
When $\alpha$ is non-separating in $S$, we have (a).
When $\alpha$ is separating in $S$, we have (b).}\label{fig-hbp-pf}
\end{center}
\end{figure}
If both $c$ and $d$ were 1-HBPs, then since $c$ and $d$ contain $\alpha$, the curve $\alpha$ would be non-separating, $c^1$ would encircle $\alpha_2$ and $\partial_2$, and $d^1$ would encircle $\alpha_1$ and $\partial_1$. 
It turns out that $e^1$ encircles $\alpha_2$ and $\partial_2$ and cannot be disjoint from $a^1$. 
This is a contradiction. 
If $c$ were a 2-HBP, then $c^1$ would have to encircle $\alpha_1$, $\partial_1$ and $\partial_2$. 
It follows that each of $d^1$ and $e^1$ encircles either $\alpha_1$ and $\partial_1$ or $\alpha_1$ and $\partial_2$, and that $d^1$ and $e^1$ cannot be disjoint and distinct. 
This is also a contradiction. 
By symmetry, one can deduce a contradiction if $d$ is assumed to be a 2-HBP.
\end{proof}

\begin{lem}\label{lem-nhbp-pen-neq}
Let $S=S_{g, 2}$ be a surface with $g\geq 2$. 
Let $(a, b_1, c_1, c_2, b_2)$ be a 5-tuple defining a pentagon $\Pi$ in $\calcp(S)$ such that
\begin{itemize}
\item any vertex of $\Pi$ corresponds to an HBP; and
\item any of the edges $\{ a, b_1\}$, $\{ a, b_2\}$ and $\{ c_1, c_2\}$ is rooted.
\end{itemize}
Let $\alpha_1$, $\alpha_2$ and $\beta$ denote the root curves of the three edges in the second condition, respectively (see Figure \ref{fig-pen-pf} (a)). 
Then the equality $\alpha_1=\alpha_2$ holds, and $\alpha_1$ and $\beta$ are disjoint, but not HBP-equivalent.
\end{lem}

\begin{proof}
We define curves $b_1^1$, $b_2^1$, $c_1^1$ and $c_2^1$ so that we have $b_1=\{ \alpha_1, b_1^1\}$, $b_2=\{ \alpha_2, b_2^1\}$, $c_1=\{ \beta, c_1^1\}$ and $c_2=\{ \beta, c_2^1\}$.
It follows from $i(b_1, c_1)=i(b_2, c_2)=0$ that $i(\alpha_1, \beta)=i(\alpha_2, \beta)=0$.
\begin{figure}
\begin{center}
\includegraphics[width=12cm]{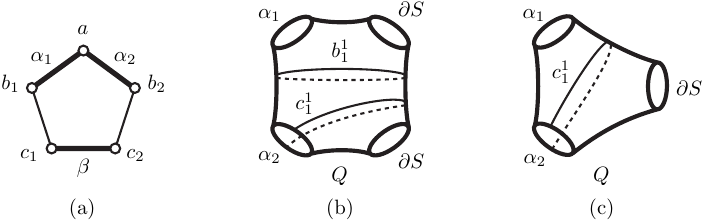}
\caption{}\label{fig-pen-pf}
\end{center}
\end{figure}
Let $Q$ denote the holed sphere cut off by $a$ from $S$.

If $\alpha_1$ and $\alpha_2$ were distinct, then we would have the equality $a=\{ \alpha_1, \alpha_2\}$.
Since $i(c_1, a)\neq 0$ and $i(c_2, a)\neq 0$, we have $i(c_1^1, \alpha_2)\neq 0$ and $i(c_2^1, \alpha_1)\neq 0$.
Suppose that $a$ is a 2-HBP. 
It then follows that $b_1^1$ and $b_2^1$ are curves in $Q$ separating the two boundary components of $Q$ that correspond to $\alpha_1$ and $\alpha_2$. 
Since $i(c_1^1, \alpha_2)\neq 0$ and $i(c_1^1, b_1^1)=i(c_1^1, \alpha_1)=0$, the intersection $c_1^1\cap Q$ consists of mutually isotopic, essential simple arcs in $Q$ (see Figure \ref{fig-pen-pf} (b)). 
Since $i(c_2^1, \alpha_1)\neq 0$ and $i(c_2^1, b_2^1)=i(c_2^1, \alpha_2)=0$, the intersection $c_2^1\cap Q$ also consists of mutually isotopic, essential simple arcs in $Q$. 
It however follows from $i(c_1, c_2)=0$ that the equality $b_1^1=b_2^1$ has to hold because $b_1^1$ is a boundary component of a regular neighborhood of the union $(c_1^1\cap Q)\cup \alpha_2$ and a similar property holds for $b_2^1$. 
This contradicts $i(b_1, b_2)\neq 0$. 

We next suppose that $a$ is a 1-HBP. 
The component $Q$ is then a pair of pants containing $\alpha_1$ and $\alpha_2$ as boundary components.
Since we have $i(c_1^1, \alpha_2)\neq 0$, $i(c_2^1, \alpha_1)\neq 0$ and $i(c_1^1, \alpha_1)=i(c_2^1, \alpha_2)=0$, any component of $c_1^1\cap Q$ and any component of $c_2^1\cap Q$ have to intersect (see Figure \ref{fig-pen-pf} (c)). 
This contradicts $i(c_1, c_2)=0$.

We have shown the equality $\alpha_1=\alpha_2$, and denote the curve by $\alpha$. 
We prove the latter assertion of the lemma.
By Lemma \ref{lem-no-pent}, we have $\alpha \neq \beta$. 
Assuming that $\alpha$ and $\beta$ are HBP-equivalent, we deduce a contradiction. 
It then follows that $b_1$ and $c_1$ are equivalent and have a common curve $\gamma$.
If $\gamma =\alpha$, then we have $c_1=\{ \alpha, \beta \}$, and this contradicts $i(c_1, b_2)\neq 0$. 
Similarly, if $\gamma =\beta$, then we have $b_1=\{ \alpha, \beta \}$, and this contradicts $i(b_1, c_2)\neq 0$.
We thus have $\gamma \neq \alpha$ and $\gamma \neq \beta$.
This implies the equalities $b_1=\{ \alpha, \gamma \}$ and $c_1=\{ \beta, \gamma \}$.
It follows from $i(c_2, \alpha)=i(c_2, \gamma)=0$ that $i(c_2, b_1)=0$.
This is a contradiction.
\end{proof}

As an application of the last two lemmas, we prove the following property on a superinjective map from $\calcp_n(S)$ into $\calcp(S)$ in the case of $p=2$.
Subsequently, we prove an analogous property in the case of $p\geq 3$.

\begin{lem}\label{lem-hbp-neq-p2}
Let $S=S_{g, 2}$ be a surface with $g\geq 2$, and let $\phi \colon \calcp_n(S)\to \calcp(S)$ be a superinjective map.
Let $a$ and $b$ be non-separating 1-HBPs in $S$ such that
\begin{itemize}
\item $a$ and $b$ are disjoint; and
\item if we choose a curve $\alpha$ in $a$ and a curve $\beta$ in $b$, then $S_{\{ \alpha, \beta \}}$ is connected.
\end{itemize}
Then the HBPs $\phi(a)$ and $\phi(b)$ are not equivalent.
\end{lem}

\begin{proof}
Using Figure \ref{fig-nhbp-pen}, we can find a pentagon $\Pi$ in $\calcp_n(S)$ defined by a 5-tuple $(x, y_1, z_1, z_2, y_2)$ such that
\begin{itemize}
\item $y_1=a$ and $z_1=b$;
\item any vertex of $\Pi$ corresponds to an HBP; and
\item any of the edges $\{ x, y_1\}$, $\{ x, y_2\}$ and $\{ z_1, z_2\}$ is rooted.
\end{itemize}
Applying Lemma \ref{lem-nhbp-pen-neq} to $\phi(\Pi)$, we obtain the lemma.
\end{proof}

\begin{lem}\label{lem-hbp-neq-p3}
Let $S=S_{g, p}$ be a surface with $g\geq 2$ and $p\geq 3$.
Then the following assertions hold:
\begin{enumerate}
\item Let $\phi \colon \calcp_n(S)\to \calcp(S)$ be a superinjective map.
If $a$ and $b$ are non-separating HBPs in $S$ which are disjoint, but not equivalent, then the HBPs $\phi(a)$ and $\phi(b)$ are not equivalent.
\item Let $\psi \colon \calcp_s(S)\to \calcp(S)$ be a superinjective map.
If $a$ and $b$ are separating HBPs in $S$ which are disjoint, but not equivalent, then the HBPs $\psi(a)$ and $\psi(b)$ are not equivalent.
\end{enumerate}
\end{lem}

\begin{proof}
We prove assertion (i).
Assertion (ii) is proved along a verbatim argument.

Let $a$ and $b$ be non-separating HBPs in $S$ which are disjoint, but not equivalent.
Replacing $a$ and $b$ by HBPs equivalent to themselves, we may assume that $a$ is a 1-HBP, and that there exists a simplex of $\calc(S)$, $s=\{ \beta_1,\ldots, \beta_p\}$, such that 
\begin{itemize}
\item any curve in $s$ is disjoint from $a$;
\item any two curves in $s$ are HBP-equivalent; and
\item $b$ is an HBP of two curves in $s$.
\end{itemize}
Let $\sigma$ be the simplex of $\calcp_n(S)$ consisting of all HBPs of two curves in $s$.
For $j=1,\ldots, p$, we define $\sigma_j$ as the rooted subsimplex of $\sigma$ that consists of $p-1$ HBPs and whose root curve is equal to $\beta_j$.
Let $\gamma_j$ denote the root curve of $\phi(\sigma_j)$.
By Lemma \ref{lem-root-curve}, the equality $\phi(\{ \beta_j, \beta_k\})=\{ \gamma_j, \gamma_k\}$ holds for any distinct $j, k=1,\ldots, p$.

If $\phi(a)$ were equivalent to $\phi(b)$, then those two HBPs would be contained in a simplex of $\calcp(S)$ of maximal dimension.
By Proposition \ref{prop-max}, for any $j=1,\ldots, p$, there exists an HBP $b_j$ in $\sigma_j$ such that $\phi(a)$ and $\phi(b_j)$ share a curve.
Suppose that for some $j$, this shared curve is not equal to $\gamma_j$.
Pick an HBP $b_j'$ in $\sigma_j$ distinct from $b_j$, which exists because $p\geq 3$.
The inclusion $\phi(b_j)\subset \phi(a)\cup \phi(b_j')$ then holds.
This is a contradiction because for any $j$, there exists an HBP $c$ with $i(c, a)=i(c, b_j')=0$ and $i(c, b_j)\neq 0$.
It thus turns out that for any $j=1,\ldots, p$, $\phi(a)$ and $\phi(b_j)$ share $\gamma_j$.
This is a contradiction because $\phi(a)$ consists of two curves and we have $p\geq 3$.
Assertion (i) is proved.
\end{proof}

In the rest of this section, we discuss topological types of vertices preserved under any superinjective map from $\calcp_n(S)$ into $\calcp(S)$.

\begin{lem}\label{lem-phi-nshbp}
Let $S=S_{g, p}$ be a surface with $g\geq 2$ and $p\geq 2$. 
Then any superinjective map from $\calcp_n(S)$ into $\calcp(S)$ preserves non-separating HBPs.
\end{lem}

\begin{proof}
The following argument appears in the proof of \cite[Lemma 3.14]{kida-yama}.
Let $s=\{\alpha_0, \alpha_1,\ldots,\alpha_p\}$ be the collection of non-separating curves in Figure \ref{fig-nshbp}.
\begin{figure}
\begin{center}
\includegraphics[width=6cm]{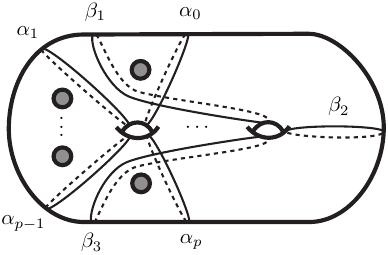}
\caption{}\label{fig-nshbp}
\end{center}
\end{figure}
We define $\sigma$ as the simplex of $\calcp_n(S)$ of maximal dimension that consists of HBPs of two curves in $s$.
Let $\beta_1$, $\beta_2$ and $\beta_3$ be the curves in $S$ in Figure \ref{fig-nshbp}. 
Note that $\beta_2$ is not HBP-equivalent to any curve of $s$, and that $S_{\{ \alpha_j, \beta_2\}}$ is connected for any $j=0,\ldots, p$. 
By Lemma \ref{lem-root-curve}, there exist curves $\gamma_j$ and $\delta_k$ in $S$ with
\[\phi(\{ \alpha_{j_1}, \alpha_{j_2}\})=\{ \gamma_{j_1}, \gamma_{j_2}\},\quad \phi(\{ \beta_{k_1}, \beta_{k_2}\})=\{ \delta_{k_1}, \delta_{k_2}\}\]
for any distinct $j_1, j_2=0,\ldots, p$ and any distinct $k_1, k_2=1, 2, 3$.
We put $t=\{ \gamma_0,\ldots, \gamma_p\}$.
Since $\{ \delta_1, \delta_2\}$ intersects $\gamma_0$ and is disjoint from any curve in $t\setminus \{ \gamma_0\}$ and since $\{ \delta_2, \delta_3\}$ intersects $\gamma_p$ and is disjoint from any curve in $t\setminus \{ \gamma_p\}$, we see that $\delta_2$ is disjoint from any curve in $t$.
By Lemmas \ref{lem-hbp-neq-p2} and \ref{lem-hbp-neq-p3}, $\delta_2$ is not HBP-equivalent to any curve of $t$.
It follows that $\{ \delta_2, \delta_1\}$ is a 1-HBP in $S$, and that there exists a $(p-1)$-HBP of two curves in $t\setminus \{ \gamma_0\}$ such that the other curves in it are contained in the holed sphere cut off by that $(p-1)$-HBP from $S$.
Similarly, $\{ \delta_2, \delta_3\}$ is a 1-HBP in $S$, and there exists a $(p-1)$-HBP of two curves in $t\setminus \{ \gamma_p\}$ such that the other curves in it are contained in the holed sphere cut off by that $(p-1)$-HBP from $S$.
It thus follows that $\{ \gamma_0, \gamma_p\}$ is a $p$-HBP in $S$.

We now assume that any curve of $t$ is separating in $S$.
If $\delta_2$ lies in the component of $S_{\{ \gamma_0, \gamma_p\}}$ of positive genus that contains $\gamma_0$ as a boundary component, then the HBP $\{ \delta_2, \delta_3\}$ cannot intersect $\gamma_p$, kept disjoint from any curve in $t\setminus \{ \gamma_p\}$. 
This is a contradiction.
Similarly, we can deduce a contradiction if we assume that $\delta_2$ lies in the component of $S_{\{ \gamma_0, \gamma_p\}}$ of positive genus that contains $\gamma_p$ as a boundary component.
The curve $\delta_2$ does not lie in the component of $S_{\{ \gamma_0, \gamma_p\}}$ of genus $0$ because any curve in that component either is HBP-equivalent to $\gamma_0$ or is an HBC in $S$.
We thus proved that any curve of $t$ is non-separating in $S$.
\end{proof}

\begin{lem}\label{lem-phi-nskhbp}
Let $S=S_{g, p}$ be a surface with $g\geq 2$ and $p\geq 2$.
Then for any integer $k$ with $1\leq k\leq p$, any superinjective map from $\calcp_n(S)$ into $\calcp(S)$ preserves non-separating $k$-HBPs.
\end{lem}

\begin{proof}
Pick a non-separating 1-HBP $a$.
There exists a simplex $s=\{ \beta_1,\ldots, \beta_p\}$ of $\calc(S)$ such that
\begin{itemize}
\item any curve in $s$ is disjoint from $a$;
\item any two curves in $s$ are HBP-equivalent; and
\item the surface obtained by cutting $S$ along $\beta_1$ and a curve of $a$ is connected.
\end{itemize}
Let $\sigma$ be the simplex of $\calcp_n(S)$ consisting of all HBPs of two curves in $s$.
It follows that $\phi(a)$ is a 1-HBP because $\phi(a)$ and any HBP in $\phi(\sigma)$ are not equivalent by Lemmas \ref{lem-hbp-neq-p2} and \ref{lem-hbp-neq-p3}, and because $\phi(\sigma)$ consists of all HBPs of two curves in a collection of mutually HBP-equivalent $p$ curves.

We have shown that $\phi$ preserves non-separating 1-HBPs.
Let $k$ be an integer with $2\leq k\leq p$.
For any non-separating $k$-HBP $c$, there exist curves $\gamma_0,\ldots, \gamma_k$ such that the equality $c=\{ \gamma_0, \gamma_k\}$ holds and for any $j$, $\{ \gamma_j, \gamma_{j+1}\}$ is a 1-HBP.
It follows that $\phi(\{ \gamma_j, \gamma_{j+1}\})$ is a 1-HBP.
Using Lemma \ref{lem-root-curve}, we see that $\phi(c)$ is a $k$-HBP.
\end{proof}

\begin{lem}\label{lem-phi-jhbc}
Let $S=S_{g, p}$ be a surface with $g\geq 2$ and $p\geq 2$.
Then for any integer $j$ with $2\leq j\leq p$, any superinjective map from $\calcp_n(S)$ into $\calcp(S)$ preserves $j$-HBCs.
\end{lem}

We first show this lemma when $p=2$, using the following observation on root curves of edges of pentagons.

\begin{lem}[\ci{Lemma 4.6}{kida-yama}]\label{lem-pent-root}
Let $S=S_{g, 2}$ be a surface with $g\geq 2$, and let $(a, b_1, c_1, c_2, b_2)$ be a $5$-tuple defining a pentagon in $\calcp_n(S)$ such that
\begin{itemize}
\item $b_1$ and $b_2$ are 2-HBPs; $c_1$ and $c_2$ are 1-HBPs; and
\item any of the three edges $\{ b_1, c_1\}$, $\{ c_1, c_2\}$ and $\{ c_2, b_2\}$ is rooted. 
\end{itemize}
Then the root curves of the three edges in the second condition are equal.
\end{lem}

\begin{proof}[Proof of Lemma \ref{lem-phi-jhbc} when $p=2$]
Pick a 2-HBC $a$ in $S$, and choose a pentagon $\Pi$ containing $a$ and drawn as in Figure \ref{fig-nhbp-c-pen}.
\begin{figure}
\begin{center}
\includegraphics[width=10cm]{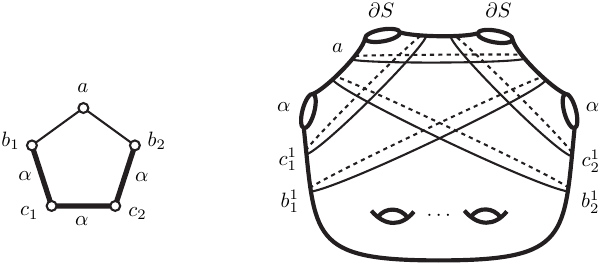}
\caption{A pentagon in $\calcp_n(S_{g, 2})$ with $b_j=\{ \alpha, b_j^1\}$ and $c_j=\{ \alpha, c_j^1\}$ for $j=1, 2$.}\label{fig-nhbp-c-pen}
\end{center}
\end{figure}
Let $(a, b_1, c_1, c_2, b_2)$ denote the 5-tuple defining $\Pi$.
We note that $b_1$ and $b_2$ are 2-HBPs and that $c_1$ and $c_2$ are 1-HBPs. 
By Lemmas \ref{lem-phi-nskhbp} and \ref{lem-pent-root}, the root curves of the three edges $\{ \phi(b_1), \phi(c_1)\}$, $\{ \phi(c_1), \phi(c_2)\}$ and $\{ \phi(c_2), \phi(b_2)\}$ are equal. 
Let $\beta$ denote the common curve. 

If $\phi(a)$ were not a 2-HBC, then $\phi(a)$ would be a 1-HBP and share a curve in $S$ with each of $\phi(b_1)$ and $\phi(b_2)$. 
We define $b_1^2$ and $b_2^2$ as the curves of $\phi(b_1)$ and $\phi(b_2)$ distinct from $\beta$, respectively.
Since we have $i(b_1^2, b_2^2)\neq 0$, $\phi(a)$ has to contain $\beta$.
This contradicts Lemma \ref{lem-no-pent}.
It therefore turns out that $\phi(a)$ is a 2-HBC, and the lemma follows.
\end{proof}

Before giving a proof of Lemma \ref{lem-phi-jhbc} when $p\geq 3$, let us prove the following:

\begin{lem}\label{lem-n-well-p3}
Let $S=S_{g, p}$ be a surface with $g\geq 2$ and $p\geq 3$, and let $\phi \colon \calcp_n(S)\to \calcp(S)$ be a superinjective map.
For any non-separating curve $\alpha$ in $S$, there exists a unique non-separating curve $\beta$ in $S$ such that for any HBP $a$ in $S$ containing $\alpha$, the HBP $\phi(a)$ contains $\beta$.
\end{lem}

\begin{proof}
Pick a non-separating curve $\alpha$ in $S$.
Choosing two disjoint and distinct HBPs $a$, $b$ in $S$ containing $\alpha$, we define $\beta$ as the root curve of the rooted edge $\{ \phi(a), \phi(b)\}$.
The curve $\beta$ depends only on $\alpha$ thanks to Lemma \ref{lem-rooted} and the following fact: For any two rooted edges $\{ a_1, b_1\}$ and $\{ a_2, b_2\}$ of $\calcp_n(S)$ whose root curves are equal to $\alpha$, there exists a sequence of rooted 2-simplices of $\calcp_n(S)$, $\sigma_1,\ldots, \sigma_m$, such that $a_1, b_1\in \sigma_1$; $a_2, b_2\in \sigma_m$; and $\sigma_j\cap \sigma_{j+1}$ is an edge of $\calcp_n(S)$ for any $j=1,\ldots, m-1$.
Existence of such a sequence follows from \cite[Lemmas 4.5 and 5.1]{kida-yama}.
Uniqueness of $\beta$ follows from injectivity of $\phi$.
\end{proof}

\begin{proof}[Proof of Lemma \ref{lem-phi-jhbc} when $p\geq 3$]
Let $\alpha$ be a $j$-HBC in $S$.
We prove that $\phi(\alpha)$ is a $j$-HBC in $S$, by induction on $j$.

Let us assume $j=2$.
Assuming that $\phi(\alpha)$ is an HBP, we deduce a contradiction.
Pick a simplex $s=\{ \beta_1,\ldots, \beta_p\}$ of $\calc(S)$ such that
\begin{itemize}
\item any curve of $s$ is non-separating in $S$ and is disjoint from $\alpha$; and
\item $\{ \beta_1, \beta_2\}$ is a 2-HBP in $S$, and $\{ \beta_k, \beta_{k+1}\}$ is a 1-HBP in $S$ for any $k=2,\ldots, p-1$
\end{itemize}
(see Figure \ref{fig-max-pf}).
\begin{figure}
\begin{center}
\includegraphics[width=5.5cm]{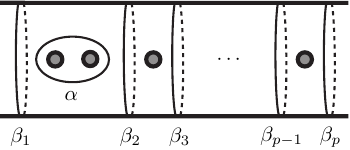}
\caption{}\label{fig-max-pf}
\end{center}
\end{figure}
By Lemma \ref{lem-root-curve}, there exists a simplex $t=\{ \gamma_1,\ldots, \gamma_p\}$ of $\calc(S)$ with $\phi(\{ \beta_k, \beta_l\})=\{ \gamma_k, \gamma_l\}$ for any distinct $k, l=1,\ldots, p$.
By Lemma \ref{lem-phi-nskhbp}, $\{ \gamma_1, \gamma_2\}$ is a 2-HBP in $S$, and $\{ \gamma_k, \gamma_{k+1}\}$ is a 1-HBP for any $k=2,\ldots, p-1$.
The HBP $\phi(\alpha)$ is equivalent to the HBP $\{ \gamma_1, \gamma_p\}$ because the latter is a $p$-HBP and is disjoint from $\phi(\alpha)$.
By Proposition \ref{prop-max}, $\phi(\alpha)$ contains at least one curve in $t$, say $\gamma_{k_0}$.
Pick $k_1\in \{ 1, \ldots, p\}\setminus \{ k_0\}$, and choose a curve $\delta$ in $S$ such that
\begin{itemize}
\item $i(\delta, \beta_{k_0})\neq 0$ and $i(\delta, \alpha)=i(\delta, \beta_k)=0$ for any $k\in \{ 1,\ldots, p\}\setminus \{ k_0\}$; and
\item $\{ \beta_{k_1}, \delta \}$ is an HBP in $S$.
\end{itemize}
It follows that $\phi(\{ \beta_{k_1}, \delta\})$ is disjoint from $\phi(\alpha)$ and intersects $\{ \gamma_{k_1}, \gamma_{k_0}\}$.
The HBP $\phi(\{ \beta_{k_1}, \delta\})$ contains $\gamma_{k_1}$ by Lemma \ref{lem-n-well-p3}.
The curve of $\phi(\{ \beta_{k_1}, \delta\})$ distinct from $\gamma_{k_1}$ thus intersects $\gamma_{k_0}$.
This is a contradiction because $\gamma_{k_0}$ belongs to $\phi(\alpha)$.

We have proved that $\phi(\alpha)$ is an HBC.
Since $\phi(\alpha)$ is disjoint from the 2-HBP $\{ \gamma_1, \gamma_2\}$ and the 1-HBP $\{ \gamma_k, \gamma_{k+1}\}$ for any $k=2,\ldots, p-1$, it has to be a 2-HBC.

We next assume $j\geq 3$.
Pick a $(j-1)$-HBC $\beta$ in $S$ disjoint from $\alpha$ and contained in the holed sphere cut off by $\alpha$ from $S$.
The hypothesis of the induction implies that $\phi(\beta)$ is also a $(j-1)$-HBC in $S$.
Let $Q$ denote the component of $S_{\beta}$ of positive genus, and let $R$ denote the component of $S_{\phi(\beta)}$ of positive genus.
The map $\phi$ induces a superinjective map from $\calcp_n(Q)$ into $\calcp(R)$.
Since $\alpha$ is a 2-HBC in $Q$, the curve $\phi(\alpha)$ is a 2-HBC in $R$ by Lemma \ref{lem-phi-jhbc} in the case of $p=2$.
It thus turns out that $\phi(\alpha)$ is a $j$-HBC in $S$.
\end{proof}

The following lemma is an immediate consequence of Lemmas \ref{lem-phi-nshbp} and \ref{lem-phi-jhbc}.

\begin{lem}\label{lem-bn-b}
Let $S=S_{g, p}$ be a surface with $g\geq 2$ and $p\geq 2$, and let $\phi \colon \calcp_n(S)\to \calcp(S)$ be a superinjective map.
Then the inclusion $\phi(\calcp_n(S))\subset \calcp_n(S)$ holds.
\end{lem}


\section{Superinjective maps from $\calcp_s(S)$ into itself}\label{sec-bs-bs}

We discuss topological types of vertices preserved under any superinjective map from $\calcp_s(S)$ into itself.
Throughout this section, we mean by an HBP a separating one, unless otherwise stated.

We mean by a {\it hexagon} in $\calcp(S)$ a subgraph of $\calcp(S)$ consisting of six vertices $v_1,\ldots, v_6$ with $i(v_k,v_{k+1})=0$, $i(v_k,v_{k+2})\ne 0$ and $i(v_k,v_{k+3})\ne 0$ for any $k$ mod $6$.
In this case, let us say that the hexagon is defined by the $6$-tuple $(v_1,\ldots,v_6)$.
A hexagon in $\calcp_s(S)$ (resp.\ $\calcp_n(S)$) is defined as a hexagon in $\calcp(S)$ any of whose vertices belongs to $\calcp_s(S)$ (resp.\ $\calcp_n(S)$).
Examples of hexagons in $\calcp(S)$ are drawn in Figures \ref{fig-shbp-hex}, \ref{fig-nhbp-c-hex} and \ref{fig-shbp-c-hex}.
\begin{figure}
\begin{center}
\includegraphics[width=12cm]{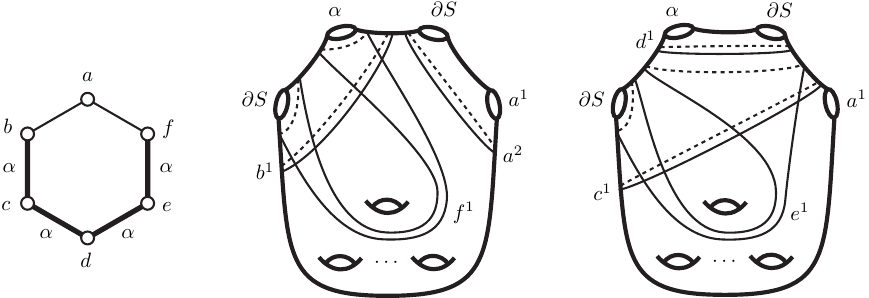}
\caption{A hexagon in $\calcp_s(S_{g, 2})$ with $g\geq 3$ and $a=\{ a^1, a^2\}$, $b=\{ \alpha, b^1\}$, $c=\{ \alpha, c^1\},\ldots, f=\{ \alpha, f^1\}$.}\label{fig-shbp-hex}
\end{center}
\end{figure}

\begin{prop}\label{prop-s-hex-root}
Let $S=S_{g, 2}$ be a surface with $g\geq 2$. Let $(a, b, c, d, e, f)$ be a 6-tuple defining a hexagon in $\calcp_s(S)$ such that
\begin{itemize}
\item any of $a,\ldots, f$ is an HBP; and
\item any of the four edges $\{ b, c\}$, $\{ c, d\}$, $\{ d, e\}$ and $\{ e, f\}$ is rooted.
\end{itemize}
Then the following assertions hold:
\begin{enumerate}
\item The HBP $a$ is equivalent to neither $b$ nor $f$. In particular, we have $g\geq 3$.
\item Any of $a$, $b$, $d$ and $f$ is a 1-HBP, and any of $c$ and $e$ is a 2-HBP.
\end{enumerate}
\end{prop}

Note that the hexagon in Figure \ref{fig-shbp-hex} satisfies the assumption in this proposition. 
Let us summarize elementary properties on separating HBPs, which will be used to prove the proposition. 
One can readily check the following two lemmas.

\begin{lem}\label{lem-shbp-eq}
Let $S=S_{g, 2}$ be a surface with $g\geq 2$. 
Let $x$ and $y$ be HBPs in $S$ with $\{ x, y\}$ an edge of $\calcp_s(S)$. 
Then the following two assertions hold:
\begin{enumerate}
\item If $x$ and $y$ are not equivalent, then $x$ and $y$ are 1-HBPs (see Figure \ref{fig-shbp-pf} (a)).
\begin{figure}
\begin{center}
\includegraphics[width=10cm]{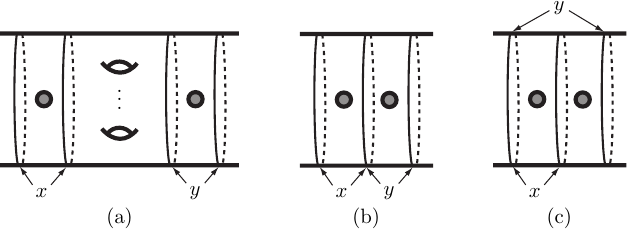}
\caption{}\label{fig-shbp-pf}
\end{center}
\end{figure}
\item If $x$ and $y$ are equivalent, then the edge $\{ x, y\}$ is rooted, and either $x$ and $y$ are 1-HBPs or one of $x$ and $y$ is a 1-HBP and another is a 2-HBP. 
The former is the case if and only if the root curve of $\{ x, y\}$ separates the two components of $\partial S$ (see Figure \ref{fig-shbp-pf} (b) and (c)).
\end{enumerate}
\end{lem}

\begin{lem}\label{lem-shbp-sq}
Let $S=S_{g, 2}$ be a surface with $g\geq 2$. The following assertions hold:
\begin{enumerate}
\item If $x$, $y$ and $z$ are 1-HBPs in $S$ such that $\{ x, y\}$ and $\{ y, z\}$ are rooted edges of $\calcp_s(S)$, then the root curves of $\{ x, y\}$ and $\{ y, z\}$ are equal.
\item If $y$ is a 1-HBP in $S$ and $x$ and $z$ are 2-HBPs in $S$ such that $\{ x, y\}$ and $\{ y, z\}$ are rooted edges of $\calcp_s(S)$, then the root curves of $\{ x, y\}$ and $\{ y, z\}$ are equal.
\item If $x$, $y$, $z$ and $w$ are mutually distinct 1-HBPs in $S$ such that $\{ x, y\}$, $\{ y, z\}$ and $\{ z, w\}$ are rooted edges of $\calcp_s(S)$, then we have $i(w, x)=0$ and the edge $\{ w, x\}$ is rooted. 
\end{enumerate}
\end{lem}

\begin{proof}[Proof of Proposition \ref{prop-s-hex-root} (i)]
Let $\Pi$ be the hexagon in $\calcp_s(S)$ defined by the 6-tuple $(a, b, c, d, e, f)$.
Assuming that the former assertion of assertion (i) is not true, we deduce a contradiction.
Note that in this case, $a$ is equivalent to both $b$ and $f$ because the six HBPs of $\Pi$ associate the same curve in $\bar{S}$, the closed surface obtained by attaching disks to all components of $\partial S$, under the inclusion of $S$ into $\bar{S}$.
It then follows that the edges $\{ a, b\}$ and $\{ f, a\}$ are also rooted.
We do not hence need to specify $a$, and deduce a contradiction in the following four cases: (1) $\Pi$ contains no 2-HBP; (2) $\Pi$ contains exactly one 2-HBP; (3) $\Pi$ contains exactly two 2-HBPs; and (4) $\Pi$ contains three 2-HBPs.

Cases (1) and (2) are impossible by Lemma \ref{lem-shbp-sq} (iii).
In case (3), we first assume that $a$ and $d$ are 2-HBPs.
Let $Q_1$ denote the holed sphere cut off by $a$ from $S$.
The curves in $b$ and $f$ that do not belong to $a$ intersect and thus fill $Q_1$.
They belong to $c$ and $e$, respectively, by Lemma \ref{lem-shbp-eq} (ii).
The equality $a=d$ thus has to hold.
This is a contradiction.
We next assume that $c$ and $e$ are 2-HBPs. By Lemma \ref{lem-shbp-sq} (ii), $c$, $d$ and $e$ share a curve, denoted by $\alpha$. 
Similarly, by Lemma \ref{lem-shbp-sq} (i), $b$, $a$ and $f$ share a curve, denoted by $\beta$.
By Lemma \ref{lem-shbp-eq} (ii), $\alpha$ does not separate the two components of $\partial S$, but $\beta$ does.
We define $c^1$, $d^1$ and $e^1$ as the curves in $c$, $d$ and $e$ distinct from $\alpha$, respectively.
It follows from $i(b, c)=i(d, c)=0$ that $\beta$ and $d^1$ are curves in the holed sphere, denoted by $Q_2$, cut off by $c$ from $S$.
Since $b$ consists of $\beta$ and a curve in $c$ and since we have $i(b, d)\neq 0$, the curves $\beta$ and $d^1$ intersect and fill $Q_2$.
This contradicts the existence of the curve $e^1$ satisfying $i(e^1, \beta)=i(e^1, d)=0$ and $i(e^1, c^1)\neq 0$.

In case (4), we may assume that $a$, $c$ and $e$ are 2-HBPs. 
By Lemma \ref{lem-shbp-sq} (ii), the three HBPs in each of the triplets $\{ c, d, e\}$, $\{ e, f, a\}$ and $\{ a, b, c\}$ share a curve in $S$, and we denote it by $\alpha$, $\beta$ and $\gamma$, respectively.
Note that $\alpha$, $\beta$ and $\gamma$ are mutually disjoint.
If $\alpha$, $\beta$ and $\gamma$ were mutually distinct, then we would have the equalities $a=\{ \beta, \gamma \}$, $c=\{ \gamma, \alpha \}$ and $e=\{ \alpha, \beta \}$, and these three HBPs would form a 2-simplex of $\calcp_s(S)$.
This is a contradiction.
If $\alpha =\beta \neq \gamma$, then we would have $a=c$ and obtain a contradiction.
A similar argument implies the equality $\alpha =\beta =\gamma$.
We define $a^1,\ldots, f^1$ as the curves of $a,\ldots, f$ distinct from $\alpha$, respectively.

Let $R$ denote the component of $S_{\alpha}$ containing $\partial S$, which contains all the curves $a^1,\ldots, f^1$.
We have the three 3-HBCs $a^1$, $c^1$ and $e^1$ in $R$ encircling the two components of $\partial S$ and the boundary component of $R$, denoted by $\partial_{\alpha}$, that corresponds to $\alpha$.
The other three 2-HBCs $b^1$, $d^1$ and $f^1$ in $R$ encircle one component of $\partial S$ and $\partial_{\alpha}$.
Without loss of generality, we may assume that $b^1$ and $d^1$ encircle $\partial_{\alpha}$ and the same component of $\partial S$, denoted by $\partial_1$.
Another component of $\partial S$ is denoted by $\partial_2$.
Let $\tilde{R}$ denote the surface obtained from $R$ by attaching a disk to $\partial_2$.
We have the simplicial map $\rho \colon \calc(R)\to \calc^*(\tilde{R})$ associated with the inclusion of $R$ into $\tilde{R}$, where $\calc^*(\tilde{R})$ is the simplicial cone over $\calc(\tilde{R})$ with the cone point $\ast$.
Note that $\rho^{-1}(\ast)$ consists of all 2-HBCs in $R$ encircling $\partial_2$ and another boundary component of $R$.

The two disjoint curves $a^1$ and $b^1$ in $R$ are isotopic in $\tilde{R}$ because $a^1$ encircles $\partial_1$, $\partial_2$ and $\partial_{\alpha}$, and $b^1$ encircles $\partial_1$ and $\partial_{\alpha}$.
We thus have the equality $\rho(a^1)=\rho(b^1)$.
Similarly, we have $\rho(b^1)=\rho(c^1)$.
Replacing $b^1$ by $d^1$, we obtain the equality $\rho(c^1)=\rho(d^1)=\rho(e^1)$.
We therefore have the equality $\rho(a^1)=\rho(b^1)=\rho(c^1)=\rho(d^1)=\rho(e^1)$.

If $f^1$ encircles $\partial_1$ and $\partial_{\alpha}$, then we have the equality $\rho(f^1)=\rho(a^1)$.
This contradicts \cite[Theorem 7.1]{kls} asserting that the inverse image of a curve in $\tilde{R}$ under $\rho$ is a simplicial tree.
If $f^1$ encircles $\partial_2$ and $\partial_{\alpha}$, then the equality $\rho(a^1)=\rho(e^1)$ implies the equality $a^1=e^1$ because $a^1$ and $e^1$ are disjoint from $f^1$ and lie in the component of $R_{f^1}$ that does not contain $\partial_2$.
This is also a contradiction.

We proved the former assertion of assertion (i).
If $g=2$, then any two disjoint and separating HBPs in $S$ are equivalent.
The existence of $\Pi$ therefore implies $g\geq 3$.
The latter assertion of assertion (i) follows.
\end{proof}

\begin{proof}[Proof of Proposition \ref{prop-s-hex-root} (ii)]
Assertion (i) and Lemma \ref{lem-shbp-eq} (i) imply that $a$, $b$ and $f$ are 1-HBPs.
By Lemma \ref{lem-shbp-sq} (iii), at least one of $c$, $d$ and $e$ is a 2-HBP.
Assuming that exactly one of $c$, $d$ and $e$ is a 2-HBP, we deduce a contradiction.
This completes the proof.

We first assume that $c$ is a 2-HBP.
By Lemma \ref{lem-shbp-sq} (i), $d$, $e$ and $f$ share a curve, denoted by $\alpha$.
The curve $\alpha$ separates the two components of $\partial S$.
It follows that $d^1$ and $f^1$, the curves of $d$ and $f$ distinct from $\alpha$, respectively, are contained in the same component of $S_{\alpha}$, and that the two curves of $a$ are contained in another component of $S_{\alpha}$.
The equality $i(a, d)=0$ therefore holds.
This is a contradiction.
By symmetry, we can also deduce a contradiction if we assume that $e$ is a 2-HBP.

We next assume that $d$ is a 2-HBP. 
Let $\alpha$, $\beta$, $\gamma$ and $\delta$ denote the root curves of $\{ b, c\}$, $\{ c, d\}$, $\{ d, e\}$ and $\{ e, f\}$, respectively.
By Lemma \ref{lem-shbp-eq} (ii), we have $\alpha \neq \beta$ and $\gamma \neq \delta$, and the equalities $c=\{ \alpha, \beta \}$ and $e=\{ \gamma, \delta \}$ thus hold.
Since $d$ contains $\beta$ and $\gamma$ and since we have $i(c, e)\neq 0$, the two curves $\alpha$ and $\delta$ intersect and fill the holed sphere cut off by $d$ from $S$.
This contradicts the existence of the 1-HBP $a$ disjoint from $\alpha$ and $\delta$.
\end{proof}

We now turn our attention to superinjective maps from $\calcp_s(S)$ into itself.

\begin{lem}\label{lem-psi-khbp}
Let $S=S_{g, p}$ be a surface with $g\geq 2$ and $p\geq 2$.
Then for any integer $k$ with $1\leq k\leq p$, any superinjective map from $\calcp_s(S)$ into itself preserves $k$-HBPs.
\end{lem}

We first prove this lemma when $p=2$. 

\begin{proof}[Proof of Lemma \ref{lem-psi-khbp} when $p=2$]
If $g\geq 3$, then the lemma is obtained by using the hexagon in Figure \ref{fig-shbp-hex} and Proposition \ref{prop-s-hex-root}.
If $g=2$, then pick a superinjective map $\psi \colon \calcp_s(S)\to \calcp_s(S)$ and the hexagon in $\calcp_s(S)$ drawn in Figure \ref{fig-shbp-c-hex}, which is defined by a 6-tuple $(a, c_1, b_1, c_3, b_2, c_2)$ such that
\begin{itemize}
\item $a$ is a 2-HBC; $b_1$ and $b_2$ are 1-HBPs; $c_1$, $c_2$ and $c_3$ are 2-HBPs; and
\item any of the four edges $\{ c_1, b_1\}$, $\{ b_1, c_3\}$, $\{ c_3, b_2\}$ and $\{ b_2, c_2\}$ is rooted.
\end{itemize}
By Lemma \ref{lem-hbp}, any of $\psi(b_1)$, $\psi(b_2)$, $\psi(c_1)$, $\psi(c_2)$ and $\psi(c_3)$ is an HBP.
By Lemma \ref{lem-rooted}, any of $\{ \psi(c_1), \psi(b_1)\}$, $\{ \psi(b_1), \psi(c_3)\}$, $\{ \psi(c_3), \psi(b_2)\}$ and $\{ \psi(b_2), \psi(c_2)\}$ is rooted.
By Proposition \ref{prop-s-hex-root} (i), $\psi(a)$ is not an HBP and is thus a 2-HBC.
It follows that $\psi$ preserves 2-HBCs.
Since any HBP in $S$ disjoint from a 2-HBC is a 2-HBP, the map $\psi$ also preserves 2-HBPs.
Since any HBP in $S$ disjoint and distinct from a 2-HBP is a 1-HBP, the map $\psi$ also preserves 1-HBPs.
\end{proof}

In the rest of the proof of Lemma \ref{lem-psi-khbp}, we assume $p\geq 3$ and prove it in the cases of $g=2$ and $g\geq 3$ separately.

\begin{proof}[Proof of Lemma \ref{lem-psi-khbp} when $g=2$ and $p\geq 3$]
Let $\psi \colon \calcp_s(S)\to \calcp_s(S)$ be a superinjective map.
Note that any two disjoint and separating HBPs in $S$ are equivalent by the assumption $g=2$.

We claim that $\psi$ preserves 2-HBCs.
Let $\alpha$ be a 2-HBC in $S$.
Choose a simplex $s=\{ \beta_1,\ldots, \beta_p\}$ of $\calc(S)$ disjoint from $\alpha$ and consisting of mutually HBP-equivalent and separating curves in $S$.
Let $\sigma$ denote the simplex of $\calcp_s(S)$ consisting of all HBPs of two curves in $s$. 
If $\psi(\alpha)$ were an HBP, then it would be equivalent to any HBP in $\psi(\sigma)$ by the assumption $g=2$.
We can deduce a contradiction along the proof of Lemma \ref{lem-hbp-neq-p3}.
It thus turns out that $\psi(\alpha)$ is an HBC.
The existence of $\psi(\sigma)$ implies that $\psi(\alpha)$ is a 2-HBC.
Our claim follows.

By induction on $p$, we prove that for any integer $k$ with $1\leq k\leq p$, the map $\psi$ preserves $k$-HBPs.
The case of $p=2$ is already proved.
The following argument appears in the proof of \cite[Lemma 3.15]{kida-yama}.
Pick a simplex $\tau$ of $\calcp_s(S)$ of maximal dimension.
Let $\{ \gamma_0,\ldots, \gamma_p\}$ denote the collection of curves in HBPs of $\tau$ so that $\{ \gamma_j, \gamma_{j+1}\}$ is a 1-HBP for any $j=0,\ldots, p-1$.
By Lemma \ref{lem-root-curve}, there exist curves $\delta_0,\ldots, \delta_p$ in $S$ with $\psi(\{ \gamma_j, \gamma_k\})=\{ \delta_j, \delta_k\}$ for any distinct $j, k=0, \ldots, p$.
Choose two distinct 2-HBCs $\epsilon_1$ and $\epsilon_2$ in $S$ contained in the holed sphere cut off by $\{ \gamma_0, \gamma_2\}$ from $S$ (see Figure \ref{fig-ep-pf} (a)).
\begin{figure}
\begin{center}
\includegraphics[width=12cm]{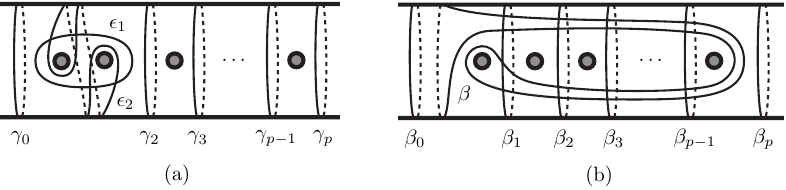}
\caption{}\label{fig-ep-pf}
\end{center}
\end{figure}
The claim in the last paragraph implies that $\psi(\epsilon_1)$ and $\psi(\epsilon_2)$ are 2-HBCs.
Let $Q$ denote the component of $S_{\epsilon_1}$ of positive genus, and let $R$ denote the component of $S_{\psi(\epsilon_1)}$ of positive genus.
We now apply the hypothesis of the induction to the superinjective map from $\calcp_s(Q)$ into $\calcp_s(R)$ induced by $\psi$.
It then follows that each of $\{ \delta_0, \delta_2\}$ and $\{ \delta_j, \delta_{j+1}\}$ for any $j=2,\ldots, p-1$ is a 1-HBP in $R$, and that $\{ \delta_0, \delta_p\}$ is a $(p-1)$-HBP in $R$.

We have two possibilities: Either $\{ \delta_0, \delta_2\}$ is a 2-HBP in $S$ or for some $j=2,\ldots, p-1$, the pair $\{ \delta_j, \delta_{j+1}\}$ is a 2-HBP in $S$.
Assuming that the latter holds, we deduce a contradiction.
We choose a curve $\gamma$ in $S$ such that $\{ \gamma_0, \gamma \}$ is an HBP in $S$; and $\gamma$ intersects $\gamma_{j+1}$ and is disjoint from $\gamma_k$ for any $k\in \{ 0,\ldots, p\} \setminus \{ j+1\}$.
Note that $\{ \gamma_0, \gamma \}$ is disjoint from $\epsilon_1$ and $\epsilon_2$.
On the other hand, $\psi(\epsilon_1)$ and $\psi(\epsilon_2)$ fill the holed sphere cut off by the 2-HBP $\{ \delta_j, \delta_{j+1}\}$ in $S$.
Since the 2-simplex of $\calcp_s(S)$ consisting of $\{ \gamma_0, \gamma \}$, $\{ \gamma_0, \gamma_1\}$ and $\{ \gamma_0, \gamma_2\}$ is rooted, the HBP $\psi(\{ \gamma_0, \gamma\})$ contains $\delta_0$.
Another curve of $\psi(\{ \gamma_0, \gamma \})$ intersects $\delta_{j+1}$, and thus intersects $\psi(\epsilon_1)$ or $\psi(\epsilon_2)$.
This is a contradiction.

We have shown that $\{ \delta_0, \delta_2\}$ is a 2-HBP in $S$, and that for any $j=2,\ldots, p-1$, $\{ \delta_j, \delta_{j+1}\}$ is a 1-HBP in $S$.
It follows that $\{ \delta_0, \delta_1\}$ and $\{ \delta_1, \delta_2\}$ are 1-HBPs in $S$, and that $\psi$ preserves $k$-HBPs for any $k$.
\end{proof}

\begin{proof}[Proof of Lemma \ref{lem-psi-khbp} when $g\geq 3$ and $p\geq 3$]
Let $\psi \colon \calcp_s(S)\to \calcp_s(S)$ be a superinjective map.
Pick a simplex $\sigma$ of $\calcp_s(S)$ of maximal dimension.
Let $s=\{ \beta_0,\ldots, \beta_p\}$ be the collection of curves in HBPs of $\sigma$ with $\{ \beta_j, \beta_{j+1}\}$ a 1-HBP for any $j=0,\ldots, p-1$.
Since we have $g\geq 3$, we may assume that the component of $S_{\sigma}$ of positive genus containing $\beta_p$ as a boundary component is of genus at least 2. 
This component of $S_{\sigma}$ is denoted by $Q$.
By Lemma \ref{lem-root-curve}, there exist curves $\gamma_0,\ldots, \gamma_p$ in $S$ with $\psi(\{ \beta_j, \beta_k\})=\{ \gamma_j, \gamma_k\}$ for any distinct $j, k=0,\ldots, p$.
Since the genus of $Q$ is at least 2, there exists a simplex $\tau$ of $\calcp_s(S)$ consisting of $p(p-1)/2$ HBPs which are mutually equivalent and any of which is disjoint from $\{ \beta_0, \beta_1\}$, but not equivalent to it.
By Lemma \ref{lem-hbp-neq-p3} (ii), the HBP $\{ \gamma_0, \gamma_1\}$ and any HBP in $\psi(\tau)$ are not equivalent.
We thus see that $\{ \gamma_0, \gamma_1\}$ is a 1-HBP and that there is a unique component of $S_{\{ \gamma_0, \gamma_1\}}$ of positive genus which has exactly one boundary component.
This component of $S_{\{ \gamma_0, \gamma_1\}}$ is denoted by $R$.

Choose a curve $\beta$ in $S$ such that $\{ \beta_0, \beta \}$ is an HBP and we have $i(\beta, \beta_p)=0$ and $i(\beta, \beta_j)\neq 0$ for any $j=1,\ldots, p-1$ (see Figure \ref{fig-ep-pf} (b)). 
Since $\{ \beta_0, \beta\}$ and $\{ \beta_0, \beta_p\}$ form a rooted edge of $\calcp_s(S)$, Lemma \ref{lem-rooted} implies that $\psi(\{ \beta_0, \beta \})$ contains $\gamma_0$ or $\gamma_p$.
Another curve, denoted by $\gamma$, of $\psi(\{ \beta_0, \beta \})$ intersects $\gamma_j$ for any $j=1,\ldots, p-1$.
It then follows that $R$ contains $\gamma_0$ as a boundary component because otherwise $R$ would contain $\gamma_1$ as a boundary component, and $\gamma$ could not intersect $\gamma_1$ and $\gamma_2$ simultaneously, kept disjoint from $\gamma_0$.

We have shown that $\{ \gamma_0, \gamma_1\}$ is a 1-HBP in $S$, and that $\partial S$ is contained in a component of $S_{\gamma_0}$.
Pick a simplex $\rho$ of $\calcp_s(S)$ consisting of $(p-1)(p-2)/2$ HBPs which are mutually equivalent and any of which is disjoint from $\{ \beta_0, \beta_2\}$, but not equivalent to it.
It follows that $\{ \gamma_0, \gamma_2\}$ is a 2-HBP in $S$ because it is disjoint from $\psi(\rho)$ and is equivalent to no HBP in $\psi(\rho)$ by Lemma \ref{lem-hbp-neq-p3} (ii).
Repeating this argument, we see that for any $j=1,\ldots, p$, the pair $\{ \gamma_0, \gamma_j\}$ is a $j$-HBP in $S$.
It follows that for any $j, k=0,\ldots, p$ with $j<k$, the pair $\{ \gamma_j, \gamma_k\}$ is a $(k-j)$-HBP in $S$.
\end{proof}

\begin{lem}\label{lem-psi-jhbc}
Let $S=S_{g, p}$ be a surface with $g\geq 2$ and $p\geq 2$.
Then for any integer $j$ with $2\leq j\leq p$, any superinjective map from $\calcp_s(S)$ into itself preserves $j$-HBCs.
\end{lem}

\begin{proof}[Proof of Lemma \ref{lem-psi-jhbc} when $p=2$]
Let $\psi \colon \calcp_s(S)\to \calcp_s(S)$ be a superinjective map. 
It follows from Lemma \ref{lem-psi-khbp} that the image of the hexagon $\Pi$ in Figure \ref{fig-shbp-c-hex} under $\psi$ contains three 2-HBPs and four rooted edges.
Proposition \ref{prop-s-hex-root} implies that the image of the 2-HBC in $\Pi$ has to be a 2-HBC.
This proves that $\psi$ preserves 2-HBCs.
\end{proof}

The proof of Lemma \ref{lem-psi-jhbc} when $p\geq 3$ is a verbatim translation of that of Lemma \ref{lem-phi-jhbc} when $p\geq 3$.
In place of Lemma \ref{lem-n-well-p3}, we use the following lemma, which can be verified by using \cite[Lemmas 4.10 and 5.3]{kida-yama}.

\begin{lem}
Let $S=S_{g, p}$ be a surface with $g\geq 2$ and $p\geq 3$, and let $\psi \colon \calcp_s(S)\to \calcp_s(S)$ be a superinjective map.
For any separating curve $\alpha$ in $S$ that is not an HBC in $S$, there exists a unique separating curve $\beta$ in $S$ that is not an HBC in $S$, such that for any HBP $a$ in $S$ containing $\alpha$, the HBP $\psi(a)$ contains $\beta$.
\end{lem}


\section{Hexagons in $\calcp_n(S_{g, 2})$}\label{sec-bn}

In this section, we focus on the case of $p=2$, and discuss topological types of hexagons in $\calcp_n(S_{g, 2})$.
We then show that any superinjective map $\phi$ from $\calcp_n(S_{g, 2})$ into itself sends the link of any 2-HBP $a$ onto the link of $\phi(a)$.
Results in this section will also be used in subsequent sections.
Throughout this section, we put $S=S_{g, 2}$ with $g\geq 2$, and mean by an HBP a non-separating one unless otherwise stated, because we mainly deal with $\calcp_n(S)$.

\begin{prop}\label{prop-n-hex}
Let $(a, c_1, b_1, c_3, b_2, c_2)$ be a 6-tuple defining a hexagon in $\calcp_n(S)$ such that
\begin{itemize}
\item $a$ is a 2-HBC; $b_1$ and $b_2$ are 1-HBPs; $c_1$, $c_2$ and $c_3$ are 2-HBPs; and
\item any of the four edges $\{ c_1, b_1\}$, $\{ b_1, c_3\}$, $\{ c_3, b_2\}$ and $\{ b_2, c_2\}$ is rooted.
\end{itemize}
Then the following two assertions hold:
\begin{enumerate}
\item The root curves of the four edges in the second condition are equal.
\item We have $i(b_1, b_2)=i(b_2, a)=i(a, b_1)=2$.
\end{enumerate}
\end{prop}

The hexagon in Figure \ref{fig-nhbp-c-hex} satisfies the assumption in this proposition.
\begin{figure}
\begin{center}
\includegraphics[width=12cm]{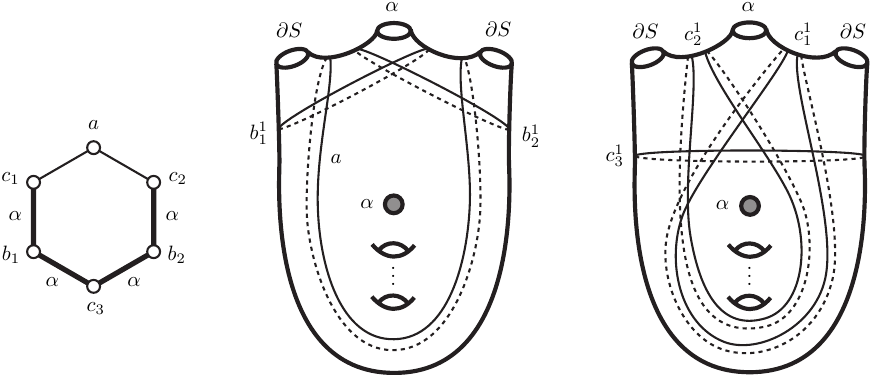}
\caption{A hexagon in $\calcp_n(S_{g, 2})$ with $b_j=\{ \alpha, b_j^1\}$ and $c_k=\{ \alpha, c_k^1\}$ for $j=1, 2$ and $k=1, 2, 3$.}\label{fig-nhbp-c-hex}
\end{center}
\end{figure}

\begin{proof}[Proof of Proposition \ref{prop-n-hex} (i)]
Let $\alpha$, $\beta$, $\gamma$ and $\delta$ denote the root curves of the edges $\{ c_1, b_1\}$, $\{ b_1, c_3\}$, $\{ c_3, b_2\}$ and $\{ b_2, c_2\}$, respectively (see Figure \ref{fig-nhex-pf} (a)).
\begin{figure}
\begin{center}
\includegraphics[width=12cm]{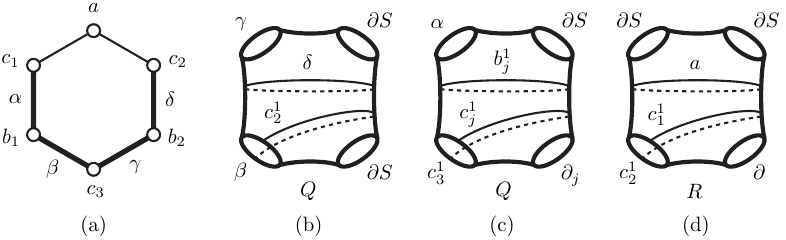}
\caption{}\label{fig-nhex-pf}
\end{center}
\end{figure}
We first assume $\beta \neq \gamma$, and deduce a contradiction.
We have $c_3=\{ \beta, \gamma \}$, and denote by $Q$ the holed sphere cut off by $c_3$ from $S$.
We define a curve $c_2^1$ so that $c_2=\{ \delta, c_2^1\}$.
We assume that any two of these curves intersect minimally unless they are isotopic.

Suppose that the equality $\alpha =\beta$ holds. 
If we had $\gamma =\delta$, then we would have $i(a, c_3)=0$ and a contradiction.
We thus have $\gamma \neq \delta$ and $b_2=\{ \gamma, \delta \}$.
Note that $\delta$ is a curve in $Q$.
Since $c_2^1$ intersects $\beta$ and is disjoint from $\delta$ and $\gamma$, the intersection $c_2^1\cap Q$ consists of mutually isotopic, essential simple arcs in $Q$ each of which cuts off an annulus containing exactly one component of $\partial S$ (see Figure \ref{fig-nhex-pf} (b)).
It is then impossible to find a position of $a$ because $a$ is a 2-HBC disjoint from $\beta =\alpha$ and $c_2$.

We thus proved $\alpha \neq \beta$.
Similarly, we have $\gamma \neq \delta$ and the equalities $b_1=\{ \alpha, \beta \}$, $c_3=\{ \beta, \gamma \}$ and $b_2=\{ \gamma, \delta \}$.
It turns out that $\alpha$ and $\delta$ intersect in $Q$ and thus fill $Q$.
Since $a$ is disjoint from $\alpha$ and $\delta$, it is impossible that $a$ is a 2-HBC.

We have shown the equality $\beta =\gamma$. 
In what follows, we deduce a contradiction in the following three cases: (1) $\alpha$, $\beta$ and $\delta$ are mutually distinct; (2) either $\alpha =\beta \neq \delta$ or $\alpha \neq \beta =\delta$; and (3) $\alpha =\delta \neq \beta$.

In case (1), we have $b_1=\{ \alpha, \beta \}$ and $b_2=\{ \beta, \delta \}$.
It then follows that $\alpha$ and $\delta$ fill the holed sphere cut off by $c_3$ from $S$.
Since $a$ is disjoint from $\alpha$ and $\delta$, it is impossible that $a$ is a 2-HBC. 
In case (2), we have either $b_2=\{ \alpha, \delta \}$ or $b_1=\{ \alpha, \delta \}$. 
Each of the two equalities contradicts the condition that $a$ is disjoint from $\alpha$ and $\delta$, but intersects $b_1$ and $b_2$. 
In case (3), we have the equality $b_1=b_2=\{ \alpha, \beta \}$ and a contradiction.
\end{proof}

\begin{proof}[Proof of Proposition \ref{prop-n-hex} (ii)]
Let $\alpha$ denote the root curve of the edges $\{ c_1, b_1\}$, $\{ b_1, c_3\}$, $\{ c_3, b_2\}$ and $\{ b_2, c_2\}$.
We define curves $b_j^1$ and $c_k^1$ so that $b_j=\{ \alpha, b_j^1\}$ and $c_k=\{ \alpha, c_k^1\}$ for $j=1, 2$ and $k=1, 2, 3$.
We assume that any two of these curves intersect minimally.

Let $Q$ denote the holed sphere cut off by $c_3$ from $S$.
Note that $b_1^1$ and $b_2^1$ are curves in $Q$.
For any $j=1, 2$, using $b_j^1$, we see that the intersection $c_j^1\cap Q$ consists of mutually isotopic, essential simple arcs in $Q$ each of which cuts off an annulus containing exactly one component of $\partial S$, denoted by $\partial_j$, and joins two points of $c_3^1$ (see Figure \ref{fig-nhex-pf} (c)).
Let $l_j$ denote a component of $c_j^1\cap Q$.
It follows from $b_1^1\neq b_2^1$ that $l_1$ and $l_2$ are not isotopic.
If $l_1$ and $l_2$ could not be isotoped so that they are disjoint, then the union of a subarc of $l_1$ and a subarc of $l_2$ would be a simple closed curve in $Q$ isotopic to $\partial_1$.
This is a contradiction because $a$ is a 2-HBC in $S$ disjoint from $c_1$ and $c_2$.
It thus turns out that $l_1$ and $l_2$ are non-isotopic and can be isotoped so that they are disjoint.
The equality $i(b_1, b_2)=i(b_1^1, b_2^1)=2$ now follows.

Replacing the 6-tuple $(a, c_1, b_1, c_3, b_2, c_2)$ with the 6-tuple $(b_1, c_3, b_2, c_2, a, c_1)$, we show $i(b_2, a)=2$.
Let $R$ denote the holed sphere cut off by $c_2^1$ from $S_{\alpha}$.
Let $\partial$ denote the boundary component of $S_{\alpha}$ that corresponds to $\alpha$ and is contained in $R$.
The holed sphere cut off by $b_2^1$ from $S_{\alpha}$ contains $\partial$ because $b_2^1$ is disjoint from $c_2^1$.
Repeating this argument, we see that the holed sphere cut off by each of $c_3^1$, $b_1^1$ and $c_1^1$ from $S_{\alpha}$ contains $\partial$.
Since $a$ is a curve in $R$ disjoint from $c_1^1$, the intersection $c_1^1\cap R$ consists of mutually isotopic, essential simple arcs in $R$ each of which cuts off an annulus containing $\partial$ and joins two points of $c_2^1$ (see Figure \ref{fig-nhex-pf} (d)).
Since $b_2^1$ is a curve in $R$ disjoint from $c_3^1$, the intersection $c_3^1\cap R$ also consists of mutually isotopic, essential simple arcs in $R$.
Let $r_1$ and $r_3$ be components of $c_1^1\cap R$ and $c_3^1\cap R$, respectively.
It follows from $a\neq b_2^1$ that $r_1$ and $r_3$ are not isotopic.
If $r_1$ and $r_3$ could not be isotoped so that they are disjoint, then the union of a subarc of $r_1$ and a subarc of $r_3$ would be a simple closed curve in $R$ isotopic to $\partial$.
This is a contradiction because $b_1^1$ is a 2-HBC in $S_{\alpha}$ which is disjoint from $c_1^1$ and $c_3^1$ and cuts off a pair of pants containing $\partial$ from $S_{\alpha}$.
It follows that $r_1$ and $r_3$ are non-isotopic and can be isotoped so that they are disjoint.
The equality $i(b_2, a)=i(b_2^1, a)=2$ is obtained.

By symmetry, we obtain the equality $i(a, b_1)=2$.
\end{proof}

As an application of the last proposition, surjectivity of a superinjective map on the link of a 2-HBP is verified.

\begin{lem}\label{lem-n2hbp-lk}
Let $\phi \colon \calcp_n(S)\to \calcp_n(S)$ be a superinjective map.
Then for any 2-HBP $b$ in $S$, we have the equality
\[\phi(\lk(b))=\lk(\phi(b)),\]
where for a vertex $c$ of $\calcp_n(S)$, $\lk(c)$ denotes the link of $c$ in $\calcp_n(S)$.
\end{lem}

Before proving this lemma, we recall the simplicial graph associated to $S_{0, 4}$.

\medskip

\noindent {\bf Graph $\cal{F}(X)$.} For a surface $X$ homeomorphic to $S_{0, 4}$, we define a simplicial graph $\cal{F}(X)$ so that the set of vertices of $\cal{F}(X)$ is $V(X)$ and two elements $\alpha, \beta \in V(X)$ are connected by an edge of $\cal{F}(X)$ if and only if we have $i(\alpha, \beta)=2$.

\medskip

This graph is known to be isomorphic to the Farey graph (see \cite[Section 3.2]{luo}).
It is shown that any injective simplicial map from the Farey graph into itself is surjective.
To prove Lemma \ref{lem-n2hbp-lk}, we also need the following:

\begin{lem}\label{lem-n-well}
Let $\phi \colon \calcp_n(S)\to \calcp_n(S)$ be a superinjective map. 
Suppose that for $k=1,2,3,4$, we have a non-separating HBP $a_k=\{\alpha, \alpha_k\}$ in $S$ such that $\{ a_1, a_2\}$ and $\{ a_3, a_4\}$ are edges of $\calcp_n(S)$. 
Then the root curves of the two edges $\{ \phi(a_1), \phi(a_2)\}$ and $\{ \phi(a_3), \phi(a_4)\}$ of $\calcp_n(S)$ are equal. 
\end{lem}

\begin{proof}
Since we have already shown that $\phi$ preserves 2-HBCs, 1-HBPs and 2-HBPs, respectively, the proof of \cite[Lemma 4.7]{kida-yama} is now valid for the present setting.
\end{proof}

\begin{proof}[Proof of Lemma \ref{lem-n2hbp-lk}]
Let $b$ be a 2-HBP in $S$.
Define curves $\beta_1$, $\beta_2$, $\gamma_1$ and $\gamma_2$ so that the equalities $b=\{ \beta_1, \beta_2\}$ and $\phi(b)=\{ \gamma_1, \gamma_2\}$ hold. 
By Lemma \ref{lem-n-well}, we may assume that for any $j=1, 2$ and any rooted edge $e$ of $\calcp_n(S)$ whose root curve is equal to $\beta_j$, the root curve of $\phi(e)$ is equal to $\gamma_j$. 
By Lemmas \ref{lem-phi-nskhbp} and \ref{lem-phi-jhbc}, $\phi$ induces a map from the union of the set of HBCs in $\lk(b)$ and the set of HBPs in $\lk(b)$ containing $\beta_1$ into the union of the set of HBCs in $\lk(\phi(b))$ and the set of HBPs in $\lk(\phi(b))$ containing $\gamma_1$.
Applying Proposition \ref{prop-n-hex} (ii), we conclude that this map induces an injective simplicial map between the Farey graphs associated to the holed spheres cut off by $b$ and by $\phi(b)$ from $S$.
Since such a simplicial map is surjective, the set $\phi(\lk(b))$ contains all HBCs in $\lk(\phi(b))$ and all HBPs in $\lk(\phi(b))$ containing $\gamma_1$.
In a similar way, we can show that $\phi(\lk(b))$ contains all HBPs in $\lk(\phi(b))$ containing $\gamma_2$.
\end{proof}

Let us denote by $\mathscr{H}_n=\mathscr{H}_n(S)$ the set of all hexagons in $\calcp_n(S)$ satisfying the assumption in Proposition \ref{prop-n-hex}.
The following lemma motivates us to relate the link of a 1-HBP $b$ in $\calcp_n(S)$ with the complex of arcs for the component of $S_b$ of positive genus.
The latter complex will be examined in Section \ref{sec-d}.

\begin{lem}\label{lem-n-hex-arc}
Let $b_1$ be a 1-HBP in $S$ and pick a curve $\alpha$ in $b_1$.
We denote by $X$ the component of $S_{b_1}$ of positive genus.
Let $c_1=\{ \alpha, c_1^1\}$ and $c_3=\{ \alpha, c_3^1\}$ be 2-HBPs in $S$ with $c_1\neq c_3$ and $i(c_1, b_1)=i(c_3, b_1)=0$.
Then there exists a hexagon in $\mathscr{H}_n$ containing $b_1$, $c_1$ and $c_3$ if and only if the defining arcs of $c_1^1$ and $c_3^1$ as curves in $X$ can be disjoint.
\end{lem}

\begin{proof}
Suppose that there exists a hexagon $\Pi$ in $\mathscr{H}_n$ containing $b_1$, $c_1$ and $c_3$.
We then have a 6-tuple $(a, c_1, b_1, c_3, b_2, c_2)$ defining $\Pi$.
For $j=1, 2$, let $b_j^1$ denote the curve of $b_j$ distinct from $\alpha$.
Let $l_a$, $l_1$ and $l_2$ denote the defining arcs of $a$, $b_1^1$ and $b_2^1$ as curves in $S_{\alpha}$, respectively.
Each of $l_1$ and $l_2$ meets exactly one of the two boundary components of $S_{\alpha}$ corresponding to $\alpha$.
Since $b_1^1$ and $b_2^1$ are curves in the holed sphere cut off by $c_3$ from $S$, the arcs $l_1$ and $l_2$ meet the same boundary component of $S_{\alpha}$ corresponding to $\alpha$.
Proposition \ref{prop-n-hex} (ii) implies that $l_a$, $l_1$ and $l_2$ can mutually be disjoint.
The defining arcs of $c_1^1$ and $c_3^1$ as curves in $X$ are $l_a\cap X$ and $l_2\cap X$, respectively, and are thus disjoint.

Conversely, suppose that the defining arcs of $c_1^1$ and $c_3^1$ as curves in $X$, denoted by $r_1$ and $r_3$, respectively, are disjoint.
Label as $\partial_1$ and $\partial_2$ the components of $\partial S$ so that $\partial_2$ is contained in the pair of pants cut off by $b_1$ from $S$.
Let $\partial_3$ denote the boundary component of $S_{\alpha}$ that corresponds to $\alpha$ and is contained in the holed sphere cut off by $c_1^1$ from $S_{\alpha}$.
The component $\partial_3$ is also contained in the holed sphere cut off by $c_3^1$ from $S_{\alpha}$.
Let $b_1^1$ denote the curve of $b_1$ distinct from $\alpha$.
Each of $r_1$ and $r_3$ connects the two component of $\partial X$ corresponding to $\partial_1$ and $b_1^1$ (see Figure \ref{fig-arc-pf}).
\begin{figure}
\begin{center}
\includegraphics[width=4cm]{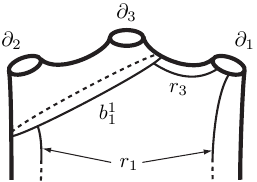}
\caption{}\label{fig-arc-pf}
\end{center}
\end{figure}

We now extend $r_1$ to an essential simple arc $l_a$ in $S_{\alpha}$ which connects $\partial_1$ and $\partial_2$.
We next extend $r_3$ to an essential simple arc $l_2$ in $S_{\alpha}$ which connects $\partial_1$ and $\partial_3$ and is disjoint from $l_a$.
Let $a$ and $b_2^1$ denote the curves in $S_{\alpha}$ defined by $l_a$ and $l_2$, respectively.
Since $l_a$ and $l_2$ are disjoint, there exists a curve $c_2^1$ disjoint from $a$ and $b_2^1$ and with the pair $\{ \alpha, c_2^1\}$ a 2-HBP in $S$.
Such a curve is unique up to isotopy.

We check that the 6-tuple $(a, c_1, b_1, c_3, b_2, c_2)$ defines a hexagon in $\mathscr{H}_n$, where $b_2=\{ \alpha, b_2^1\}$ and $c_2=\{ \alpha, c_2^1\}$.
Let $l_1$ denote the defining arc of $b_1^1$ as a curve in $S_{\alpha}$, which connects $\partial_2$ and $\partial_3$.
Since $l_2$ connects $\partial_1$ and $\partial_3$, we have $i(b_1^1, b_2^1)\neq 0$.
Similarly, we have $i(b_1^1, a)\neq 0$ and $i(a, b_2^1)\neq 0$.
If $c_2^1$ were disjoint from $b_1^1$, then $c_2^1$ would be disjoint from $b_1^1\cup a$, and we would have $c_2^1=c_1^1$ because $c_1^1$ is a boundary component of a regular neighborhood of $b_1^1\cup a$ in $S_{\alpha}$.
We also have $c_2^1=c_3^1$ because of a similar reason.
This contradicts $c_1^1\neq c_3^1$.
We thus have $i(b_1^1, c_2^1)\neq 0$.
In particular, $c_1^1$, $c_2^1$ and $c_3^1$ are mutually distinct.
The same kind of argument shows that $i(c_3^1, a)\neq 0$ and $i(c_1^1, b_2^1)\neq 0$.
Since $c_1$, $c_2$ and $c_3$ are mutually distinct 2-HBPs in $S$ containing $\alpha$, the curves $c_1^1$, $c_2^1$ and $c_3^1$ mutually intersect.
\end{proof}


\section{Hexagons in $\calcp_s(S_{g, 2})$}\label{subsec-s-hex}\label{sec-bs}

In this section, we discuss hexagons in $\calcp_s(S_{g, 2})$, and obtain results similar to those in the previous section (see Figure \ref{fig-shbp-c-hex} for such a hexagon).
\begin{figure}
\begin{center}
\includegraphics[width=12cm]{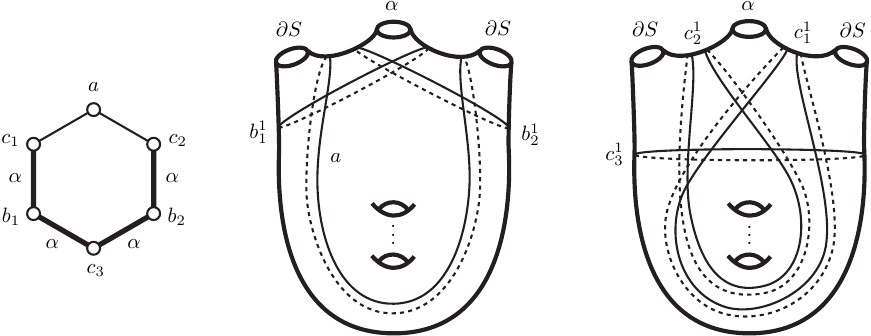}
\caption{A hexagon in $\calcp_s(S_{g, 2})$ with $b_j=\{ \alpha, b_j^1\}$ and $c_k=\{ \alpha, c_k^1\}$ for $j=1, 2$ and $k=1, 2, 3$.}\label{fig-shbp-c-hex}
\end{center}
\end{figure}
The proof is also similar, and its large part is thus omitted.
Throughout this section, we put $S=S_{g, 2}$ with $g\geq 2$, and mean by an HBP a separating one unless otherwise stated, because we mainly deal with $\calcp_s(S)$.

\begin{prop}\label{prop-s-hex}
Let $(a, c_1, b_1, c_3, b_2, c_2)$ be a 6-tuple defining a hexagon in $\calcp_s(S)$ such that
\begin{itemize}
\item $a$ is a 2-HBC; $b_1$ and $b_2$ are 1-HBPs; $c_1$, $c_2$ and $c_3$ are 2-HBPs; and
\item any of the four edges $\{ c_1, b_1\}$, $\{ b_1, c_3\}$, $\{ c_3, b_2\}$ and $\{ b_2, c_2\}$ is rooted.
\end{itemize}
Then the following two assertions hold:
\begin{enumerate}
\item The root curves of the four edges in the second condition are equal.
\item We have $i(b_1, b_2)=i(b_2, a)=i(a, b_1)=2$.
\end{enumerate}
\end{prop}

\begin{proof}
Assertion (i) is proved in \cite[Lemma 4.12]{kida-yama}.
Let $\alpha$ denote the common curve in assertion (i).
Let $T$ denote the component of $S_{\alpha}$ containing $\partial S$, whose genus is positive and less than $g$. 
For $j=1, 2$ and $k=1, 2, 3$, we define curves $b_j^1$ and $c_k^1$ so that $b_j=\{ \alpha, b_j^1\}$ and $c_k=\{ \alpha, c_k^1\}$. 
Let $Q$ denote the holed sphere cut off by $c_3^1$ from $T$.
It follows that $a$ and each $c_k^1$ are curves in $T$ and that each $b_j^1$ is a curve in $Q$.
Along an argument of the same kind as in the proof of Proposition \ref{prop-n-hex} (ii), we obtain the equality in assertion (ii).
\end{proof}

\begin{lem}\label{lem-s2hbp-lk}
Let $\psi \colon \calcp_s(S)\to \calcp_s(S)$ be a superinjective map.
Then for any 2-HBP $b$ in $S$, we have the equality
\[\psi(\lk(b))=\lk(\psi(b)),\]
where for a vertex $c$ of $\calcp_s(S)$, $\lk(c)$ denotes the link of $c$ in $\calcp_s(S)$.
\end{lem}

The proof of this lemma is a verbatim translation of the proof of Lemma \ref{lem-n2hbp-lk} once the following lemma is obtained.

\begin{lem}\label{lem-s-well}
Let $\psi \colon \calcp_s(S)\to \calcp_s(S)$ be a superinjective map. 
Suppose that for $k=1,2,3,4$, we have an HBP $a_k=\{\alpha, \alpha_k\}$ in $S$ such that $\{ a_1, a_2\}$ and $\{ a_3, a_4\}$ are edges of $\calcp_s(S)$. 
Then the root curves of the two edges $\{ \psi(a_1), \psi(a_2)\}$ and $\{ \psi(a_3), \psi(a_4)\}$ of $\calcp_s(S)$ are equal. 
\end{lem}

\begin{proof}
Since we have already shown that $\psi$ preserves 2-HBCs, 1-HBPs and 2-HBPs, respectively, the proof of \cite[Lemma 4.14]{kida-yama} is now valid for the present setting.
\end{proof}

Let us denote by $\mathscr{H}_s=\mathscr{H}_s(S)$ the set of all hexagons in $\calcp_s(S)$ satisfying the assumption in Proposition \ref{prop-s-hex}.
Along an argument of the same kind as in the proof of Lemma \ref{lem-n-hex-arc}, we can show the following:

\begin{lem}\label{lem-s-hex-arc}
Let $b_1$ be a 1-HBP in $S$ and let $\alpha$ denote the curve in $b_1$ that does not separate the two components of $\partial S$.
We denote by $Y$ the component of $S_{b_1}$ of positive genus and containing the curve of $b_1$ distinct from $\alpha$ as a boundary component.
Let $c_1=\{ \alpha, c_1^1\}$ and $c_3=\{ \alpha, c_3^1\}$ be 2-HBPs in $S$ with $c_1\neq c_3$ and $i(c_1, b_1)=i(c_3, b_1)=0$.
Then there exists a hexagon in $\mathscr{H}_s$ containing $b_1$, $c_1$ and $c_3$ if and only if the defining arcs of $c_1^1$ and $c_3^1$ as curves in $Y$ can be disjoint.
\end{lem}


\section{The complexes $\cald(X, \partial)$ and $\cald(Y)$}\label{sec-d}

We introduce natural subcomplexes of the complexes of arcs for certain surfaces, motivated by Lemmas \ref{lem-n-hex-arc} and \ref{lem-s-hex-arc}.
These two lemmas describe a relationship between the link of a 1-HBP $a$ in $\calcp_n(S)$ or $\calcp_s(S)$ and the complex of arcs for a component of $S_a$ of positive genus.
Let us recall the complex of arcs for a surface.

\medskip

\noindent {\bf Complex $\cala(S)$.} Let $S$ be a surface with non-empty boundary. 
We define $V_a(S)$ as the set of isotopy classes of essential simple arcs in $S$.
Let $\cala(S)$ denote the abstract simplicial complex whose $n$-simplices are defined as a subset $\sigma$ of $V_a(S)$ such that we have $|\sigma|=n+1$ and have mutually disjoint representatives of elements of $\sigma$.
We often identify an element of $V_a(S)$ with its representative if there is no confusion.

\medskip

\noindent {\bf Complexes $\cald(X, \partial)$ and $\cald(Y)$.} Let $X=S_{g, 3}$ be a surface with $g\geq 1$, and fix a boundary component $\partial$ of $X$.
We define $\cald(X, \partial)$ as the full subcomplex of $\cala(X)$ spanned by all vertices that correspond to arcs in $X$ connecting a component of $\partial X\setminus \partial$ with another component of $\partial X\setminus \partial$.

Let $Y=S_{g, 2}$ be a surface with $g\geq 1$.
We define $\cald(Y)$ as the full subcomplex of $\cala(Y)$ spanned by all vertices that correspond to arcs in $Y$ connecting a component of $\partial Y$ with another component of $\partial Y$.

\medskip

The aim of this section is to show the following:

\begin{prop}\label{prop-d-inj}
In the above notation, any injective simplicial map from $\cald(X, \partial)$ into itself is surjective. Moreover, any injective simplicial map from $\cald(Y)$ into itself is also surjective.
\end{prop}

The first subsection is preliminary to the proof of this proposition.
We introduce variants of the complex of arcs, and show that most of them are connected.
These variants are special ones of the complex introduced and denoted by $BZ(\Delta, \Delta^0)$ in \cite{harer}.
The proof of Proposition \ref{prop-d-inj} is presented in the second subsection.


\subsection{Connectivity of complexes of arcs}\label{subsec-a-conn}

Throughout this subsection, we fix a surface $Q=S_{g, p+1}$ with $g\geq 0$ and $p\geq 1$, and fix a boundary component $\partial_0$ of $Q$.
Let $R$ be the surface obtained from $Q$ by attaching a disk to $\partial_0$.
We label the components of $\partial Q$ other than $\partial_0$ as $\partial_1,\ldots, \partial_p$.
Suppose that we have a subset $\Delta$ of $\partial Q\setminus \partial_0$ with its decomposition $\Delta =\Delta^+\sqcup \Delta^-$ satisfying the following two conditions: For any $j=1,\ldots, p$,
\begin{itemize}
\item the two sets $\Delta^+\cap \partial_j$ and $\Delta^-\cap \partial_j$ are non-empty, are finite and have the same cardinality; and
\item along $\partial_j$, points of $\Delta^+\cap \partial_j$ and $\Delta^-\cap \partial_j$ appear alternatively.
\end{itemize}

\medskip

\noindent {\bf Complexes $\cala(Q, \Delta)$ and $\cala(R, \Delta)$.} We denote by $V_a(Q, \Delta)$ the set of isotopy classes relative to $\Delta$ of simple arcs $l$ in $Q$ such that
\begin{itemize}
\item $l$ connects a point of $\Delta^+$ with a point of $\Delta^-$, and meets $\partial Q$ only at its end points; and
\item $l$ is not homotopic relative to $\Delta$ to an arc in $\partial Q$ which meets $\Delta$ only at its end points.
\end{itemize} 
Let $\cala(Q, \Delta)$ denote the abstract simplicial complex whose $n$-simplices are defined as a subset $\sigma$ of $V_a(Q, \Delta)$ such that we have $|\sigma|=n+1$ and have representatives $l_0,\ldots, l_n$ of elements of $\sigma$ with $l_j\cap l_k=\partial l_j\cap \partial l_k$ for any distinct $j$ and $k$. 

The set $V_a(R, \Delta)$ and the complex $\cala(R, \Delta)$ are defined in the same manner after replacing $Q$ with $R$ in the last paragraph.
We often identify an element of $V_a(Q, \Delta)$ or $V_a(R, \Delta)$ with its representative if there is no confusion.

\medskip

In the rest of this subsection, we prove the following:

\begin{prop}\label{prop-a-conn}
We put $n=|\Delta^+\cap \partial_1|=|\Delta^-\cap \partial_1|$.
Then the following two assertions hold:
\begin{enumerate}
\item If $\cala(Q, \Delta)$ is not connected, then we have $(g, p)=(0, 1)$ and $n\leq 2$. 
If $(g, p, n)=(0, 1, 2)$, then $\cala(Q, \Delta)$ is of dimension 0 and consists of four vertices.
If $(g, p, n)=(0, 1, 1)$, then $\cala(Q, \Delta)=\emptyset$.
\item If $\cala(R, \Delta)$ is not connected, then we have $(g, p)=(0, 1)$ and $n\leq 3$. 
If $(g, p, n)=(0, 1, 3)$, then $\cala(R, \Delta)$ is of dimension 0 and consists of three vertices.
If $(g, p, n)=(0, 1, 2), (0, 1, 1)$, then $\cala(R, \Delta)=\emptyset$.
\end{enumerate}
\end{prop}

As already mentioned, the complexes $\cala(Q, \Delta)$ and $\cala(R, \Delta)$ are special ones of the complex introduced and denoted by $BZ(\Delta, \Delta^0)$ in \cite{harer}, where its high connectivity is discussed.
The proof however does not care its connectivity when $g$ and $p$ are small.
We hence present a direct proof of Proposition \ref{prop-a-conn}.

In the case of $(g, p)=(0, 1)$, both $\cala(Q, \Delta)$ and $\cala(R, \Delta)$ consist of finitely many vertices.
The assertion in Proposition \ref{prop-a-conn} for this case can directly be checked (see Figure \ref{fig-dzero} for the case where the complexes are of dimension 0).
\begin{figure}
\begin{center}
\includegraphics[width=10cm]{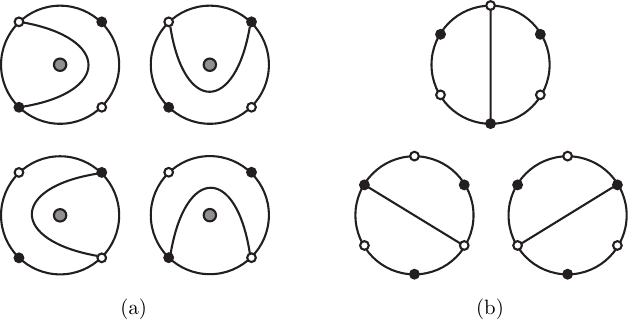}
\caption{(a) The four vertices of $\cala(Q, \Delta)$ when $(g, p, n)=(0, 1, 2)$. (b) The three vertices of $\cala(R, \Delta)$ when $(g, p, n)=(0, 1, 3)$.
White circles are points of $\Delta^+$, black circles are points of $\Delta^-$, and gray circles denote $\partial_0$.}\label{fig-dzero}
\end{center}
\end{figure}
In what follows, we therefore assume $(g, p)\neq (0, 1)$, and prove that $\cala(R, \Delta)$ is connected.
Along a similar argument, connectivity of $\cala(Q, \Delta)$ is also proved.
We thus omit its proof, making only a brief comment on it in the end of this subsection.

Let $\cala^*(R)$ denote the simplicial cone over $\cala(R)$ with its cone point $\ast$.
We have the simplicial map
\[\pi \colon \cala(R, \Delta)\to \cala^*(R)\]
defined by forgetting $\Delta$, where $\pi^{-1}(\ast)$ consists of all arcs in $\cala(R, \Delta)$ connecting two points of $\partial_j$ for some $j$ and homotopic relative to their end points to an arc in $\partial_j$.
The set $\pi^{-1}(\ast)$ may possibly be empty.

Let $u$ be an arc corresponding to a vertex of $\pi^{-1}(\cala(R))$.
We now observe that any arc corresponding to a vertex of the fiber $\pi^{-1}(\pi(u))$ is obtained by applying to $u$ the twists about the component(s) of $\partial R$ that $u$ meets.
Let us explain this fact more precisely.
We fix an orientation of $R$.
For $j=1,\ldots, p$, we put $n_j=|\Delta^+\cap \partial_j|=|\Delta^-\cap \partial_j|$, and set
\[\Delta^+\cap \partial_j=\{ x_1^j,\ldots, x_{n_j}^j\}\quad \textrm{and}\quad \Delta^-\cap \partial_j=\{ y_1^j,\ldots, y_{n_j}^j\}\]
so that $x_1^j, y_1^j, x_2^j, y_2^j,\ldots, x_{n_j}^j, y_{n_j}^j$ appear in this order along the orientation of $\partial_j$ induced by that of $R$.
We define $t_j$ as a homeomorphism of $R$ which is the identity outside a collar neighborhood $N_j$ of $\partial_j$ and satisfies the equalities
\[t_j(x_k^j)=y_k^j,\quad t_j(y_l^j)=x_{l+1}^j\quad \textrm{and}\quad t_j(y_{n_j}^j)=x_1^j\]
for any $k=1,\ldots, n_j$ and any $l=1,\ldots, n_j-1$.

Pick an arc $u$ corresponding to a vertex of $\pi^{-1}(\cala(R))$.
We first assume that the two points of $\partial u$ lie in distinct components of $\partial R$, say $\partial_1$ and $\partial_2$.
We then have the equality
\[\pi^{-1}(\pi(u))=\{\, t_1^zt_2^wu\mid z, w\in \mathbb{Z},\ z+w\in 2\mathbb{Z}\,\}.\tag{$\dagger$}\label{eq-zw}\]

We next assume that the two points of $\partial u$ lie in the same component of $\partial R$, say $\partial_1$.
Choose an arc $u_1$ in $\pi^{-1}(\pi(u))$ connecting $x_1^1$ and $y_1^1$.
We may assume that $u_1\cap \bar{N_1}$ consists of exactly two components $u_+$, $u_-$ with $x_1^1\in u_+$ and $y_1^1\in u_-$, where $\bar{N_1}$ denotes the closure of $N_1$.
Put $u_0=u_1\setminus N_1$.
For $k=2,\ldots, n_1$, we define an essential simple arc $u_k$ in $R$ as the union
\[u_k=u_+\cup u_0\cup t_1^{2(k-1)}u_-.\]
The arc $u_k$ connects $x_1^1$ and $y_k^1$, and belongs to $\pi^{-1}(\pi(u))$.
The set $\{ u_1,\ldots, u_{n_1}\}$ is a simplex of $\cala(R, \Delta)$.
We then have the equality
\[\pi^{-1}(\pi(u))=\{\, t_1^zu_k\mid z\in \mathbb{Z},\ k=1,\ldots, n_1\,\}.\tag{$\ddagger$}\label{eq-z}\]
The following two lemmas show that there exists an edge or a vertex of $\cala(R)$ whose inverse image under $\pi$ spans a connected subcomplex of $\cala(R, \Delta)$.

\begin{lem}\label{lem-uv}
Assume $g\geq 1$.
Let $\{ u, v\}$ be an edge of $\pi^{-1}(\cala(R))$ such that
\begin{itemize}
\item $u$ and $v$ are non-separating in $R$;
\item a single component of $\partial R$ contains both $\partial u$ and $\partial v$; and
\item when we cut $R$ along $u$ and obtain a connected surface, denoted by $R_u$, the arc $v$ connects the two boundary components of $R_u$ containing $u$.  
\end{itemize}
Then the full subcomplex of $\cala(R, \Delta)$ spanned by the union $\pi^{-1}(\pi(u))\cup \pi^{-1}(\pi(v))$ is connected. 
\end{lem}

\begin{proof}
We may assume that $\partial u$ and $\partial v$ are contained in $\partial_1$, and put $\partial =\partial_1$.
We use the same notation right before the lemma, and put $n=n_1$ and $t=t_1$.
We may also assume that $\partial u=\partial v=\{ x_1^1, y_1^1\}$ and that $u$ is decomposed into three arcs, $u=u_+\cup u_0\cup u_-$, so that $u_+$ and $u_-$ are the two components of $u\cap \bar{N}_1$ with $x_1^1\in u_+$ and $y_1^1\in u_-$; and we have $u_0=u\setminus N_1$.

To prove the lemma, by equation (\ref{eq-z}), it suffices to show that there exists a path in $\cala(R, \Delta)$ connecting $u$ and $tu$ and consisting of vertices in $\pi^{-1}(\pi(u))\cup \pi^{-1}(\pi(v))$.
If $n\geq 2$, then define an essential simple arc $w$ in $R$ as the union $w=u_+\cup u_0\cup t^2u_-$ (see Figure \ref{fig-ann-pf} (a)).
\begin{figure}
\begin{center}
\includegraphics[width=11cm]{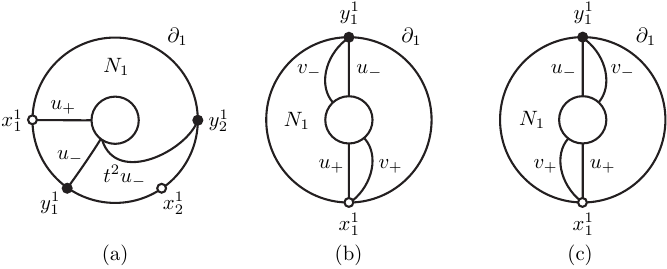}
\caption{}\label{fig-ann-pf}
\end{center}
\end{figure}
The sequence $u$, $w$, $tu$ defines a path in $\cala(R, \Delta)$.
We suppose $n=1$.
We may assume that $v$ is also decomposed into three arcs, $v=v_+\cup v_0\cup v_-$, so that $v_+$ and $v_-$ are the two components of $v\cap \bar{N}_1$ with $x_1^1\in v_+$ and $y_1^1\in v_-$; and we have $v_0=v\setminus N_1$.
The third assumption on $u$ and $v$ in the lemma implies that either case (b) or case (c) in Figure \ref{fig-ann-pf} holds.
The arc $u$ is therefore disjoint from either $tv$ or $t^{-1}v$. 
It follows that either the sequence $u$, $tv$, $tu$ or the sequence $u$, $v$, $tu$ defines a path in $\cala(R, \Delta)$.
\end{proof}

\begin{lem}\label{lem-u}
Assume $p\geq 2$.
Let $u$ be an arc corresponding to a vertex of $\cala(R, \Delta)$ and connecting distinct components of $\partial R$.
Then the full subcomplex of $\cala(R, \Delta)$ spanned by $\pi^{-1}(\pi(u))$ is connected.
\end{lem}

\begin{proof}
We may assume that $u$ connects $\partial_1$ with $\partial_2$.
Both $\{ u, t_2^2u\}$ and $\{ u, t_1^{-1}t_2u\}$ are edges of $\cala(R, \Delta)$.
The lemma then follows from equation (\ref{eq-zw}).
\end{proof}

\begin{proof}[Proof of connectivity of $\cala(R, \Delta)$ in the case of $(g, p)\neq (0, 1)$]
We note that $\cala(R)$ is connected by \cite[Theorem (a)]{hatcher}.
We also note that for any vertex $x$ of $\cala(R, \Delta)$, the image of the link of $x$ in $\cala(R, \Delta)$ under the map $\pi \colon \cala(R, \Delta)\to \cala^*(R)$ contains the intersection of $\cala(R)$ with the link of $\pi(x)$ in $\cala^*(R)$.

We suppose $g\geq 1$.
There exists an edge $\{ u, v\}$ of $\cala(R, \Delta)$ satisfying the condition in Lemma \ref{lem-uv}.
Let $w$ be a vertex of $\cala(R, \Delta)$.
Since $\cala(R)$ is connected, there exists a path in $\cala(R)$ joining $\pi(w)$ to $\pi(u)$.
The fact stated in the end of the previous paragraph implies that there exists a path in $\cala(R, \Delta)$ joining $w$ to a vertex in $\pi^{-1}(\pi(u))$.
Since any vertex in $\pi^{-1}(\pi(u))$ is joined to $u$ in $\cala(R, \Delta)$ by Lemma \ref{lem-uv}, the vertices $w$ and $u$ are joined through a path in $\cala(R, \Delta)$.
Connectivity of $\cala(R, \Delta)$ follows.
If $g=0$ and $p\geq 2$, then we can prove connectivity of $\cala(R, \Delta)$ in a similar way, using Lemma \ref{lem-u} in place of Lemma \ref{lem-uv}.
\end{proof}

Let us make a comment on the proof of connectivity of $\cala(Q, \Delta)$ in the case of $(g, p)\neq (0, 1)$.
Let $\cala^*(Q)$ denote the simplicial cone over $\cala(Q)$ with its cone point $\ast$.
We then have the simplicial map $\pi \colon \cala(Q, \Delta)\to \cala^*(Q)$ defined by forgetting $\Delta$.
To prove connectivity of $\cala(Q, \Delta)$ along the proof for $\cala(R, \Delta)$, we need to know connectivity of the full subcomplex of $\cala(Q)$ spanned by $\pi(V_a(Q, \Delta))\setminus \{ \ast \}$, in place of connectivity of $\cala(R)$.
The set $\pi(V_a(Q, \Delta))\setminus \{ \ast \}$ is equal to the set of all vertices corresponding to arcs connecting two points of $\partial Q\setminus \partial_0$.
Connectivity of this full subcomplex follows from \cite[Theorem (a)]{hatcher}.


\subsection{Injections of $\cald(X, \partial)$ and $\cald(Y)$}

Let $X=S_{g, 3}$ and $Y=S_{g, 2}$ be surfaces with $g\geq 1$, and fix a boundary component $\partial$ of $X$.
In this subsection, we discuss the properties of the complexes $\cald(X, \partial)$ and $\cald(Y)$ stated in the following two lemmas, and prove Proposition \ref{prop-d-inj}.
Let $\mathcal{E}$ be a simplicial complex with $\dim \mathcal{E}=N<\infty$.
We mean by a {\it chamber} of $\mathcal{E}$ a simplex of $\mathcal{E}$ of dimension $N$.
We say that $\mathcal{E}$ is {\it chain-connected} if for any two chambers $\sigma$, $\tau$ of $\mathcal{E}$, there exists a sequence of chambers of $\mathcal{E}$, $\sigma_0, \sigma_1,\ldots, \sigma_m$, such that we have $\sigma_0=\sigma$ and $\sigma_m=\tau$ and for any $j=0, 1,\ldots, m-1$, the intersection $\sigma_j\cap \sigma_{j+1}$ is a simplex of codimension 1.

\begin{lem}\label{lem-d-x}
In the above notation, the following assertions hold:
\begin{enumerate}
\item Let $\sigma$ be a simplex of $\cald(X, \partial)$ of codimension 1.
We denote by $X_{\sigma}$ the surface obtained by cutting $X$ along arcs in $\sigma$.
Then the number of chambers of $\cald(X, \partial)$ containing $\sigma$ is 3 if the component of $X_{\sigma}$ containing $\partial$ contains exactly two arcs in $\sigma$, and otherwise that number is 4. 
\item The complex $\cald(X, \partial)$ is chain-connected.
\end{enumerate}
\end{lem}

\begin{lem}\label{lem-d-y}
In the above notation, the following assertions hold:
\begin{enumerate}
\item For any simplex $\sigma$ of $\cald(Y)$ of codimension 1, the number of chambers of $\cald(Y)$ containing $\sigma$ is 3. 
\item The complex $\cald(Y)$ is chain-connected.
\end{enumerate}
\end{lem}

Before proving these two lemmas, let us recall basic facts on punctured surfaces and ideal arcs, which are fully discussed in \cite{mosher}.
Let $T$ be a closed surface of positive genus $g$, and let $P$ be a non-empty finite subset of $T$.
The pair $(T, P)$ is then called a {\it punctured surface}.
Let $I$ denote the closed unit interval.
We mean by an {\it ideal arc} in $(T, P)$ the image of a continuous map $f\colon I\rightarrow T$ such that
\begin{itemize}
\item we have $f(\partial I)\subset P$ and $f(I\setminus \partial I)\subset T\setminus P$;
\item $f$ is injective on $I\setminus \partial I$; and
\item there exists no closed disk $D$ embedded in $T$ with $\partial D=f(I)$ and $(D\setminus \partial D)\cap P=\emptyset$.
\end{itemize}
Two ideal arcs $l_1$, $l_2$ in $(T, P)$ are called {\it isotopic} if the equality $l_1\cap P=l_2\cap P$ holds; and $l_1$ and $l_2$ are isotopic relative to $l_1\cap P$ as arcs in $(T\setminus P)\cup (l_1\cap P)$.
We mean by an {\it ideal triangulation} of $(T, P)$ is a cell division $\delta$ of $T$ such that
\begin{enumerate}
\item[(a)] the set of 0-cells of $\delta$ is $P$;
\item[(b)] any 1-cell of $\delta$ is an ideal arc in $(T, P)$; and
\item[(c)] any 2-cell of $\delta$ is a {\it triangle}, that is, it is obtained by attaching a Euclidean triangle $\tau$ to the 1-skeleton of $\delta$, mapping each vertex of $\tau$ to a 0-cell of $\delta$, and each edge of $\tau$ to a 1-cell of $\delta$.
\end{enumerate}
Let $S$ be a surface of genus $g$ with $|P|$ boundary components.
Suppose that $T$ is obtained from $S$ by shrinking each component of $\partial S$ into a point, and that $P$ is the set of points into which components of $\partial S$ are shrunken.
The natural map from $S$ onto $T$ induces the bijection from $V_a(S)$ onto the set of isotopy classes of ideal arcs in $(T, P)$.
Under this identification, a chamber of $\cala(S)$ corresponds to an ideal triangulation of $T$.

\begin{proof}[Proof of Lemma \ref{lem-d-y} (i)]
Let $Y^*$ denote the punctured surface obtained from $Y$ by shrinking each component of $\partial Y$.
We identify $V_a(Y)$ with the set of isotopy classes of ideal arcs in $Y^*$.
Any chamber $\sigma$ of $\cald(Y)$ gives rise to a {\it squaring} of $Y^*$, that is, when we cut $Y^*$ along arcs in $\sigma$, we obtain finitely many squares whose vertices are punctures of $Y^*$ and whose edges are arcs in $\sigma$.
We see that
\begin{itemize}
\item for any square $\Pi$ of the squaring, diagonal vertices of $\Pi$ correspond to the same puncture of $Y^*$, and the two vertices of any edge of $\Pi$ correspond to distinct punctures of $Y^*$ (see Figure \ref{fig-sq} (a)); and
\begin{figure}
\begin{center}
\includegraphics[width=10cm]{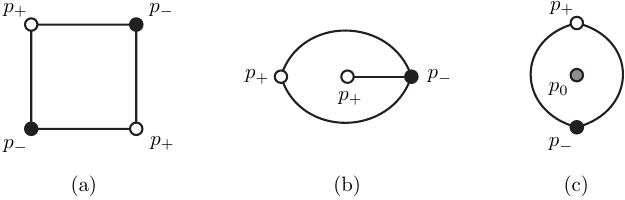}
\caption{(a) A square with $p_{\pm}$ the punctures of $Y^*$; (c) A digon with $p_{\pm}$ and $p_0$ the punctures of $X^*$.}\label{fig-sq}
\end{center}
\end{figure}
\item for any arc $l$ in $\sigma$, the two squares having $l$ as an edge are distinct.
In other words, for any square $\Pi$ of the squaring, any two distinct edges of $\Pi$ correspond to distinct arcs in $\sigma$.
\end{itemize}
The latter is proved as follows.
Let $\Pi$ denote the square in Figure \ref{fig-sq} (a).
Any two opposite edges of $\Pi$ are not identified in $Y^*$ because $Y^*$ is orientable.
If some two edges of $\Pi$ having a common vertex were identified in $Y^*$, then we would have Figure \ref{fig-sq} (b) in $Y^*$ or the figure obtained by exchanging $p_+$ and $p_-$.
When we have Figure \ref{fig-sq} (b), the square $\Pi$ is a neighborhood of $p_+$ in $Y^*$, and there can be only one ideal arc in $Y^*$ starting at $p_+$.
This is impossible.
When we have another figure, we also obtain a contradiction by replacing $p_+$ with $p_-$.

Any chamber of $\cald(Y)$ consists of $4g$ arcs, and the number of squares in the squaring associated with it is $2g$.
If we cut $Y^*$ along arcs in a simplex of $\cald(Y)$ of codimension 1, then we obtain one hexagon $H$ and $2g-2$ squares.
There exist exactly three arcs in $H$ each of which decomposes $H$ into two squares.
Lemma \ref{lem-d-y} (i) follows.
\end{proof}

\begin{proof}[Proof of Lemma \ref{lem-d-x} (i)]
Let $X^*$ denote the punctured surface obtained from $X$ by shrinking each component of $\partial X$.
Let $p_0$ denote the puncture of $X^*$ corresponding to $\partial$.
As in the proof of Lemma \ref{lem-d-y} (i), if we cut $X^*$ along arcs in a chamber of $\cald(X, \partial)$, then we obtain $2g$ squares and one digon containing $p_0$ (see Figure \ref{fig-sq} (c)).
Any chamber of $\cald(X, \partial)$ consists of $4g+1$ arcs.

Let $\tau$ be a simplex of $\cald(X, \partial)$ of codimension 1.
We cut $X^*$ along arcs in $\tau$, and obtain finitely many components.
If the component containing $p_0$ has exactly four edges, then the number of chambers containing $\tau$ is 4.
This is because there exist exactly four arcs lying in that component and corresponding to a vertex of $\cald(X, \partial)$, as drawn in Figure \ref{fig-dzero} (a), where the gray circle denotes $p_0$.
Otherwise, the component containing $p_0$ has exactly two edges, we have one hexagon, and the other components are squares.
The number of chambers containing $\tau$ is then 3.
This is because there exist exactly three arcs lying in that hexagon and corresponding to a vertex of $\cald(X, \partial)$, as drawn in Figure \ref{fig-dzero} (b).
Lemma \ref{lem-d-x} (i) follows.
\end{proof}

\begin{proof}[Proof of Lemma \ref{lem-d-y} (ii)]
We prove that $\cald(Y)$ is connected, using the technique due to Putman \cite{putman-conn} to show connectivity of a simplicial complex on which $\pmod(Y)$ acts.
Let $l$ be the arc in Figure \ref{fig-d} (b).
We pick an arc $r$ corresponding to a vertex of $\cald(Y)$, and show that $l$ and $r$ can be connected by a path in $\cald(Y)$.
We define $T$ as the set consisting of the Dehn twists about the curves in Figure \ref{fig-d} (a) and their inverses.
\begin{figure}
\begin{center}
\includegraphics[width=11cm]{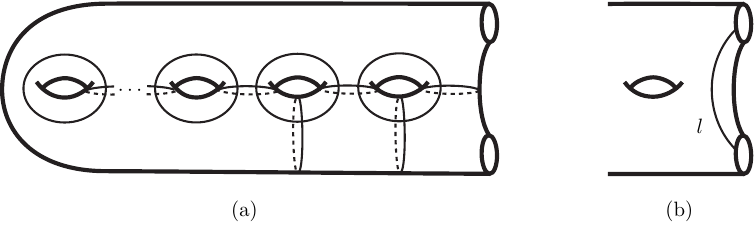}
\caption{(a) $\pmod(Y)$ is generated by Dehn twists about these curves.}\label{fig-d}
\end{center}
\end{figure}
It is known that $\pmod(Y)$ is generated by $T$ (see \cite{gervais}).
Since $l$ and $r$ are sent to each other by an element of $\pmod(Y)$, there exist elements $h_1,\ldots, h_n$ of $T$ with $r=h_1\cdots h_nl$.
We note that for any $h\in T$, either $hl=l$ or $hl$ and $l$ are disjoint.
The sequence of vertices of $\cald(Y)$,
\[l,\ h_1l,\ h_1h_2l,\ \ldots,\ h_1\cdots h_nl=r,\]
therefore forms a path in $\cald(Y)$.
Connectivity of $\cald(Y)$ follows.

To prove Lemma \ref{lem-d-y} (ii), it suffices to show that the link of any simplex of $\cald(Y)$ of codimension at least 2 is connected.
The idea to derive chain-connectivity from connectivity of such a link is due to Hatcher \cite{hatcher} and is also used in the proof of \cite[Proposition 4.7]{bm}.

Let $\sigma$ be a simplex of $\cald(Y)$ of codimension at least 2.
We identify the arc corresponding to a vertex of $\cald(Y)$ with an ideal arc in the punctured surface $Y^*$. 
Let $p_+$ and $p_-$ denote the two punctures of $Y^*$. 
We cut $Y^*$ along arcs in $\sigma$, and then obtain finitely many surfaces $Y_1^*,\ldots, Y_m^*$.
For $j=1,\ldots, m$, let $\Delta_j^+$ and $\Delta_j^-$ denote the sets of points of $Y_j^*$ corresponding to $p_+$ and $p_-$, respectively, and set $\Delta_j=\Delta_j^+\cup \Delta_j^-$.

For any $j=1,\ldots, m$, the set $\Delta_j$ with this decomposition satisfies the two conditions stated in the beginning of Section \ref{subsec-a-conn}.
Each vertex of the complex $\cala(Y_j^*, \Delta_j)$ is identified with a vertex of the link of $\sigma$ in $\cald(Y)$.
If there exist distinct $j, k=1,\ldots, m$ with both $\cala(Y_j^*, \Delta_j)$ and $\cala(Y_k^*, \Delta_k)$ non-empty, then the link of $\sigma$ is connected.
Otherwise, since $\sigma$ is of codimension at least 2, there exists a unique $j$ with $\cala(Y_j^*, \Delta_j)$ of dimension at least 1.
By Proposition \ref{prop-a-conn} (ii), $\cala(Y_j^*, \Delta_j)$ is connected.
\end{proof}

Lemma \ref{lem-d-x} (ii) can be proved similarly by using Proposition \ref{prop-a-conn} and using Figure \ref{fig-hbchbp} (a) in place of Figure \ref{fig-d} (a).

\begin{proof}[Proof of Proposition \ref{prop-d-inj}]
We first note that any vertex of $\cald(X, \partial)$ lies in a chamber of $\cald(X, \partial)$, and the same property holds for $\cald(Y)$.
Let $\psi \colon \cald(Y)\to \cald(Y)$ be an injective simplicial map.
Pick a chamber $\sigma$ of $\cald(Y)$.
Injectivity of $\psi$ and Lemma \ref{lem-d-y} (i) imply that for any face $\rho$ of $\psi(\sigma)$ of codimension 1, any chamber of $\cald(Y)$ containing $\rho$ is in the image of $\psi$.
By Lemma \ref{lem-d-y} (ii), any chamber of $\cald(Y)$ is in the image of $\psi$.
Surjectivity of $\psi$ follows.

Let $\phi \colon \cald(X, \partial)\to \cald(X, \partial)$ be an injective simplicial map.
To prove surjectivity of $\phi$, we fix the notation.
For a simplex $\tau$ of $\cald(X, \partial)$ of codimension 1, we define $n(\tau)$ as the number of chambers of $\cald(X, \partial)$ containing $\tau$.
Let $X^*$ denote the punctured surface obtained from $X$ by shrinking each component of $\partial X$.
Let $p_0$ denote the puncture of $X^*$ corresponding to $\partial$.

Let $\sigma$ be a chamber of $\cald(X, \partial)$.
Cutting $X^*$ along arcs in $\sigma$, we obtain one digon containing $p_0$ and $2g$ squares.
There are exactly two arcs in $\sigma$ whose union is the boundary of the digon.
If $l$ is any of those two arcs in $\sigma$, then cutting $X^*$ along arcs in $\sigma \setminus \{ l\}$, we obtain one square containing $p_0$ and $2g-1$ squares.
By Lemma \ref{lem-d-x} (i), we have $n(\sigma \setminus \{ l\})=4$.
For any arc $l'$ in $\sigma$ except for those two arcs, cutting $X^*$ along arcs in $\sigma \setminus \{ l'\}$, we obtain one digon containing $p_0$, one hexagon and $2g-2$ squares.
By Lemma \ref{lem-d-x} (i), we have $n(\sigma \setminus \{ l'\})=3$.
We showed that there are exactly two faces $\tau$ of $\sigma$ of codimension 1 with $n(\tau)=4$, and that any other face $\tau'$ of $\sigma$ of codimension 1 satisfies $n(\tau')=3$.

Let $\tau$ be a face of $\sigma$ of codimension 1.
Injectivity of $\phi$ implies that if $n(\tau)=4$, then $n(\phi(\tau))=4$.
It follows that both of the two faces $\rho$ of $\phi(\sigma)$ of codimension 1 with $n(\rho)=4$ is the image of some $\tau$ with $n(\tau)=4$ under $\phi$. 
Injectivity of $\phi$ again implies that if $n(\tau)=3$, then $n(\phi(\tau))=3$.
For any face $\rho$ of $\phi(\sigma)$ of codimension 1, any chamber of $\cald(X, \partial)$ containing $\rho$ is thus in the image of $\phi$.

Since $\sigma$ is an arbitrary chamber of $\cald(X, \partial)$, by Lemma \ref{lem-d-x} (ii), any chamber of $\cald(X, \partial)$ is in the image of $\phi$.
Surjectivity of $\phi$ follows.
\end{proof}


\section{Surjectivity of superinjective maps}\label{sec-surj}

We are now ready to show surjectivity of any superinjective map from $\calcp_n(S)$ into itself and any superinjective map from $\calcp_s(S)$ into itself.

\subsection{The case of $p=2$}

Let $S=S_{g, 2}$ be a surface with $g\geq 2$.
Recall that we have the simplicial map $\pi \colon \calc(S)\to \calc^*(\bar{S})$, where $\bar{S}$ is the closed surface obtained from $S$ by attaching disks to all components of $\partial S$.
For a vertex $v$ of $\calc(\bar{S})$, we define $\calc(S)_v$ as the full subcomplex of $\calc(S)$ spanned by $\pi^{-1}(v)$, which is connected by \cite[Theorem 7.1]{kls}.
We also have the simplicial maps
\[\theta_n\colon \calcp_n(S)\to \calc^*(\bar{S}),\quad \theta_s\colon \calcp_s(S)\to \calc^*(\bar{S})\]
associated with $\pi$.
For a vertex $v$ of $\calc(\bar{S})$, we define $\calcp_n(S)_v$ as the full subcomplex of $\calcp_n(S)$ spanned by $\theta_n^{-1}(v)$.
Similarly, we define $\calcp_s(S)_v$ as the full subcomplex of $\calcp_s(S)$ spanned by $\theta_s^{-1}(v)$.
The following lemma will be used in the sequel.

\begin{lem}\label{lem-fiber-conn}
Let $S=S_{g, 2}$ be a surface with $g\geq 2$.
Then
\begin{enumerate}
\item for any non-separating curve $\alpha$ in $\bar{S}$, the complex $\calcp_n(S)_{\alpha}$ is connected.
\item for any separating curve $\beta$ in $\bar{S}$, the complex $\calcp_s(S)_{\beta}$ is connected.
\end{enumerate}
\end{lem}

\begin{proof}
Pick a non-separating curve $\alpha$ in $\bar{S}$.
By \cite[Lemma 4.5]{kida-yama}, the link of any vertex of $\calc(S)_{\alpha}$ in $\calc(S)_{\alpha}$ is connected.
Combining this fact with connectivity of $\calc(S)_{\alpha}$, we obtain assertion (i).
Similarly, assertion (ii) is proved by using \cite[Lemma 4.10]{kida-yama} asserting that for any separating curve $\beta$ in $\bar{S}$, the link of any vertex of $\calc(S)_{\beta}$ in $\calc(S)_{\beta}$ is connected.
\end{proof}

We define $V_n(\bar{S})$ as the subset of $V(\bar{S})$ consisting of non-separating curves in $\bar{S}$, and define $\calc_n(\bar{S})$ as the full subcomplex of $\calc(\bar{S})$ spanned by $V_n(\bar{S})$.

We now outline the proof of surjectivity of a superinjective map $\phi \colon \calcp_n(S)\to \calcp_n(S)$.
We first show that $\phi$ sends the fiber of $\theta_n$ over a vertex of $V_n(\bar{S})$ onto the fiber of $\theta_n$ over some vertex of $V_n(\bar{S})$.
It follows that $\phi$ induces a map $\bar{\phi}$ from $V_n(\bar{S})$ into itself.
We next show that $\bar{\phi}$ defines a superinjective map from $\calc_n(\bar{S})$ into itself.
The latter map is induced by an element of $\mod^*(\bar{S})$ due to Irmak \cite{irmak-ns}, and is thus surjective.
We then conclude surjectivity of $\phi$.

\begin{thm}\label{thm-bn-surj}
Let $S=S_{g, 2}$ be a surface with $g\geq 2$.
Then any superinjective map $\phi \colon \calcp_n(S)\to \calcp_n(S)$ is surjective.
\end{thm}

\begin{proof}
Throughout this proof, we mean by an HBP in $S$ a non-separating HBP in $S$.
For a vertex $a$ of $\calcp_n(S)$, we denote by $\lk(a)$ the link of $a$ in $\calcp_n(S)$.
We denote by $\lk(a)^0$ the set of vertices in $\lk(a)$.
Let $V_n(S)$ denote the subset of $V(S)$ consisting of non-separating curves in $S$.
By Lemmas \ref{lem-root-curve} and \ref{lem-n-well}, we have a map $\Phi \colon V_n(S)\to V_n(S)$ satisfying the equality $\phi(\{ \alpha, \beta \})=\{ \Phi(\alpha), \Phi(\beta)\}$ for any HBP $\{ \alpha, \beta \}$ in $S$.

In Lemma \ref{lem-n2hbp-lk}, we have shown the equality $\phi(\lk(b))=\lk(\phi(b))$ for any 2-HBP $b$ in $S$.
Behavior of $\phi$ in the link of a 1-HBP is obtained in the following:

\begin{claim}\label{claim-bn-eq-lk}
For any 1-HBP $a$ in $S$ and any curve $\alpha$ in $a$, we have the equality
\[\phi(\{\, b\in \lk(a)^0\mid \alpha \in b \,\})=\{\, c\in \lk(\phi(a))^0\mid \Phi(\alpha)\in c\,\}.\]
\end{claim}

\begin{proof}
Let $X$ and $X'$ be the components of $S_a$ and $S_{\phi(a)}$ of positive genus, respectively.
We define $\partial$ and $\partial'$ as the boundary components of $X$ and $X'$ that correspond to $\alpha$ and $\Phi(\alpha)$, respectively.
We claim that $\phi$ induces an injective simplicial map from $\cald(X, \partial)$ into $\cald(X', \partial')$.

There is a one-to-one correspondence between essential simple arcs in $X$ connecting the two components of $\partial X \setminus \partial$ and curves in $X$ cutting off a pair of pants containing the two components of $\partial X \setminus \partial$.
The pair of $\alpha$ and such a curve in $X$ is a 2-HBP in $S$.
The same property holds for $X'$ and $\partial'$ in place of $X$ and $\partial$, respectively.
By Lemma \ref{lem-n-hex-arc}, $\phi$ induces a simplicial map from $\cald(X, \partial)$ into $\cald(X', \partial')$, which is injective because so is $\phi$.

By Proposition \ref{prop-d-inj}, this induced map is surjective.
It follows that $\phi$ sends the set of 2-HBPs in $\lk(a)$ containing $\alpha$ onto the set of 2-HBPs in $\lk(\phi(a))$ containing $\Phi(\alpha)$.
Let $\alpha'$ denote the curve in $a$ distinct from $\alpha$.
Applying the same argument to $\alpha'$ in place of $\alpha$, we see that $\phi$ sends the set of 2-HBPs in $\lk(a)$ containing $\alpha'$ onto the set of 2-HBPs in $\lk(\phi(a))$ containing $\Phi(\alpha')$.
This is equivalent to that $\phi$ sends the set of 1-HBPs in $\lk(a)$ containing $\alpha$ onto the set of 1-HBPs in $\lk(\phi(a))$ containing $\Phi(\alpha)$.
The claim follows.
\end{proof}

\begin{claim}\label{claim-bn-fiber}
For any HBP $a$ in $S$, we have the equality
\[\phi(\{\, b\in U_n(S)\mid \theta_n(b)=\theta_n(a)\,\})=\{\, c\in U_n(S)\mid \theta_n(c)=\theta_n(\phi(a))\,\},\]
where $U_n(S)$ denotes the set of vertices of $\calcp_n(S)$.
\end{claim}

\begin{proof}
Put $\alpha =\theta_n(a)$ and $\beta=\theta_n(\phi(a))$.
These are non-separating curves in $\bar{S}$.
Since $\phi$ preserves disjoint and equivalent HBPs and since $\calcp_n(S)_{\alpha}$ is connected by Lemma \ref{lem-fiber-conn} (i), the map $\phi$ induces a simplicial map from $\calcp_n(S)_{\alpha}$ into $\calcp_n(S)_{\beta}$.
By Lemma \ref{lem-n2hbp-lk} and Claim \ref{claim-bn-eq-lk}, this induced map sends the link of each vertex $u$ of $\calcp_n(S)_{\alpha}$ onto the link of $\phi(u)$.
The claim is proved.
\end{proof}

Claim \ref{claim-bn-fiber} implies that $\phi$ induces a map $\bar{\phi}$ from $V_n(\bar{S})$ into itself with $\bar{\phi}(\theta_n(a))=\theta_n(\phi(a))$ for any HBP $a$ in $S$.

\begin{claim}\label{claim-n-si}
The map $\bar{\phi}\colon V_n(\bar{S})\to V_n(\bar{S})$ defines a superinjective map from $\calc_n(\bar{S})$ into itself.
\end{claim}

\begin{proof}
For any $\alpha, \beta \in V_n(\bar{S})$, we have $i(\alpha, \beta)=0$ if and only if there exist a vertex $a$ of $\calcp_n(S)_{\alpha}$ and a vertex $b$ of $\calcp_n(S)_{\beta}$ with $i(a, b)=0$.
Let us refer this fact as $(\ast)$.

Pick $\alpha, \beta \in V_n(\bar{S})$.
If $i(\alpha, \beta)=0$, then by $(\ast)$, there exist a vertex $a$ of $\calcp_n(S)_{\alpha}$ and a vertex $b$ of $\calcp_n(S)_{\beta}$ with $i(a, b)=0$.
Simpliciality of $\phi$ implies $i(\phi(a), \phi(b))=0$.
By the definition of $\bar{\phi}$, $\phi(a)$ is a vertex of $\calcp_n(S)_{\bar{\phi}(\alpha)}$, and $\phi(b)$ is a vertex of $\calcp_n(S)_{\bar{\phi}(\beta)}$.
Applying $(\ast)$ again, we obtain $i(\bar{\phi}(\alpha), \bar{\phi}(\beta))=0$.
The map $\bar{\phi}$ is thus simplicial.

We next assume $i(\bar{\phi}(\alpha), \bar{\phi}(\beta))=0$.
By $(\ast)$, there exist a vertex $a'$ of $\calcp_n(S)_{\bar{\phi}(\alpha)}$ and a vertex $b'$ of $\calcp_n(S)_{\bar{\phi}(\beta)}$ with $i(a', b')=0$.
Claim \ref{claim-bn-fiber} implies that there exist a vertex $a$ of $\calcp_n(S)_{\alpha}$ and a vertex $b$ of $\calcp_n(S)_{\beta}$ with $\phi(a)=a'$ and $\phi(b)=b'$.
Since $\phi$ is superinjective, we have $i(a, b)=0$.
Applying $(\ast)$ again, we obtain $i(\alpha, \beta)=0$.
The map $\bar{\phi}$ is therefore superinjective. 
\end{proof}

Thanks to Irmak \cite[Theorem 1.3]{irmak-ns}, $\bar{\phi}$ is induced by an element of $\mod^*(\bar{S})$.
The map $\bar{\phi}$ is thus surjective.
Claim \ref{claim-bn-fiber} implies that the image of the map $\phi \colon \calcp_n(S)\to \calcp_n(S)$ contains all HBPs in $S$.
Since for any 2-HBP $a$ in $S$, $\phi$ sends the link of $a$ onto the link of $\phi(a)$ by Lemma \ref{lem-n2hbp-lk}, the image of $\phi$ contains all 2-HBCs in $S$.
We thus conclude that $\phi$ is surjective.
\end{proof}

We turn to showing surjectivity of a superinjective map $\psi \colon \calcp_s(S)\to \calcp_s(S)$ when $S=S_{g, 2}$ with $g\geq 2$.
For a surface $R$, we denote by $V_s(R)$ the subset of $V(R)$ consisting of separating curves in $R$, and define $\calc_s(R)$ as the full subcomplex of $\calc(R)$ spanned by $V_s(R)$.
When $g\geq 3$, we prove surjectivity of $\psi$, following the proof of Theorem \ref{thm-bn-surj} and associating to $\psi$ a superinjective map $\bar{\psi}$ from $\calc_s(\bar{S})$ into itself.
This map $\bar{\psi}$ is surjective due to the first author \cite{kida-cohop}.

When $g=2$, we cannot follow the last part of this proof because $\calc_s(\bar{S})$ is of dimension 0.
In this case, we directly show that the map $\Psi \colon V_s(S)\to V_s(S)$ induced by $\psi$ defines a superinjective map from $\calc_s(S)$ into itself.
The latter map is surjective by \cite{kida-cohop}, and thus so is $\psi$.

\begin{thm}\label{thm-bs-surj}
Let $S=S_{g, 2}$ be a surface with $g\geq 2$.
Then any superinjective map $\psi \colon \calcp_s(S)\to \calcp_s(S)$ is surjective.
\end{thm}

\begin{proof}
Throughout this proof, we mean by an HBP in $S$ a separating HBP in $S$.
For a vertex $a$ of $\calcp_s(S)$, we denote by $\lk(a)$ the link of $a$ in $\calcp_s(S)$.
We denote by $\lk(a)^0$ the set of vertices in $\lk(a)$.
By Lemmas \ref{lem-root-curve} and \ref{lem-s-well}, we have a map $\Psi \colon V_s(S)\to V_s(S)$ with $\psi(\{ \alpha, \beta \})=\{ \Psi(\alpha), \Psi(\beta)\}$ for any HBP $\{ \alpha, \beta \}$ in $S$ and with $\psi(\gamma)=\Psi(\gamma)$ for any HBC $\gamma$ in $S$.

\begin{claim}\label{claim-bs-eq-lk}
Let $a$ be a 1-HBP in $S$.
We denote by $\alpha_0$ the curve in $a$ that does not separate the two components of $\partial S$, and denote by $\alpha_1$ the curve in $a$ distinct from $\alpha_0$.
Then we have the equalities
\begin{align*}
\psi(\{\, b\in \lk(a)^0\mid \alpha_0\in b \,\})&=\{\, c\in \lk(\psi(a))^0\mid \Psi(\alpha_0)\in c\,\},\\
\psi(\{\, d\in \lk(a)^0\mid \alpha_1\in d \,\})&=\{\, e\in \lk(\psi(a))^0\mid \Psi(\alpha_1)\in e\,\}.
\end{align*}
\end{claim}

\begin{proof}
Choose a curve $\alpha_2$ in $S$ such that the pair $\{ \alpha_1, \alpha_2\}$ is a 1-HBP in $S$ disjoint and distinct from $a$.
We denote by $Y$ and $Z$ the components of $S_{\alpha_1}$ containing $\alpha_0$ and $\alpha_2$, respectively.
Similarly, we denote by $Y'$ and $Z'$ the components of $S_{\Psi(\alpha_1)}$ containing $\Psi(\alpha_0)$ and $\Psi(\alpha_2)$, respectively.
By Lemma \ref{lem-s-hex-arc}, $\psi$ induces an injective simplicial map from $\cald(Y)$ into $\cald(Y')$ and an injective simplicial map from $\cald(Z)$ into $\cald(Z')$.
Comparing the dimensions of these simplicial complexes, we see that $Y$ and $Y'$ are homeomorphic and that $Z$ and $Z'$ are homeomorphic.
By Proposition \ref{prop-d-inj}, the induced map from $\cald(Z)$ into $\cald(Z')$ is surjective.
The first equality in the claim then follows.

Let $e$ be an element in the right hand side of the second equality in the claim.
This $e$ is a 1-HBP in $S$.
Let $e^1$ denote the curve in $e$ distinct from $\Psi(\alpha_1)$.
Since the 2-HBP $\{ \Psi(\alpha_0), e^1\}$ lies in the right hand side of the first equality in the claim, there exists a curve $d^1$ in $S$ with $\{ \alpha_0, d^1\} \in \lk(a)^0$ and $\psi(\{ \alpha_0, d^1\})=\{ \Psi(\alpha_0), e^1\}$.
We then have the equality $e=\{ \Psi(\alpha_1), e^1\} =\psi(\{ \alpha_1, d^1\})$ by Lemma \ref{lem-root-curve}.
The second equality in the claim follows.
\end{proof}

Recall that we have the simplicial map $\theta_s\colon \calcp_s(S)\to \calc^*(\bar{S})$ associated with the inclusion of $S$ into $\bar{S}$.
The proof of the following claim is obtained as a verbatim translation of that of Claim \ref{claim-bn-fiber}, by exchanging symbols appropriately and using Lemma \ref{lem-s2hbp-lk} and Claim \ref{claim-bs-eq-lk} in place of Lemma \ref{lem-n2hbp-lk} and Claim \ref{claim-bn-eq-lk}.

\begin{claim}
For any HBP $a$ in $S$, we have the equality
\[\psi(\{\, b\in U_s(S)\mid \theta_s(b)=\theta_s(a)\,\})=\{\, c\in U_s(S)\mid \theta_s(c)=\theta_s(\psi(a))\,\},\]
where $U_s(S)$ denotes the set of vertices of $\calcp_s(S)$.
\end{claim}

We now obtain the map $\bar{\psi}\colon V_s(\bar{S})\to V_s(\bar{S})$ with the equality $\bar{\psi}(\theta_s(a))=\theta_s(\psi(a))$ for any HBP $a$ in $S$.
The following claim can be verified along an argument of the same kind as in the proof of Claim \ref{claim-n-si}.

\begin{claim}
The map $\bar{\psi}\colon V_s(\bar{S})\to V_s(\bar{S})$ defines a superinjective map from $\calc_s(\bar{S})$ into itself.
\end{claim}

If $g\geq 3$, then $\bar{\psi}$ is surjective by \cite[Theorem 1.1]{kida-cohop}, and surjectivity of $\psi$ is shown along the end of the proof of Theorem \ref{thm-bn-surj}.
In the rest of the proof of Theorem \ref{thm-bs-surj}, we assume $g=2$ and show that the map $\Psi \colon V_s(S)\to V_s(S)$ defines a superinjective map from $\calc_s(S)$ into itself.

\begin{claim}\label{claim-g2}
Assume $g=2$.
For any separating curve $\alpha$ in $S$ which is not an HBC in $S$, we have the equalities
\begin{align*}
\psi(\{\, b\in U_s(S)\mid i(\alpha, b)=0\,\})&=\{\, c\in U_s(S)\mid i(\Psi(\alpha), c)=0\,\},\\
\psi(\{\, b\in V_{sp}(S)\mid \alpha \in b\, \})&=\{\, c\in V_{sp}(S)\mid \Psi(\alpha)\in c\,\}.
\end{align*}
\end{claim}

\begin{proof}
We first assume that $\alpha$ does not separate the two components of $\partial S$.
Let $Q$ and $Q'$ denote the components of $S_{\alpha}$ and $S_{\Psi(\alpha)}$ containing $\partial S$, respectively, which are homeomorphic to $S_{1, 3}$.
For any separating curve $\beta$ in $Q$, either $\beta$ corresponds to an HBC in $S$ or the pair $\{ \alpha, \beta \}$ is an HBP in $S$.
The same property holds for $Q'$ and $\Psi(\alpha)$ in place of $Q$ and $\alpha$, respectively.
It follows that $\psi$ induces a superinjective map from $\calc_s(Q)$ into $\calc_s(Q')$, which is surjective by \cite[Theorem 1.1]{kida-cohop}.
The first equality in the claim is proved.
If $b=\{ b^1, b^2\}$ is an HBP in $S$ with $i(\alpha, b)=0$ and $\alpha \not\in b$, then any two of $\alpha$, $b^1$ and $b^2$ form an HBP in $S$, and we have $\Psi(\alpha)\not\in \psi(b)$ by Lemma \ref{lem-root-curve}.
The second equality in the claim thus follows from the first equality.

We next assume that $\alpha$ separates the two components of $\partial S$.
Let us denote by $R_1$ and $R_2$ the components of $S_{\alpha}$, and pick a curve $\beta$ in $R_2$ with $\{ \alpha, \beta \}$ a 1-HBP in $S$.
We define $R_1'$ and $R_2'$ the components of $S_{\Psi(\alpha)}$ with $\Psi(\beta)\in V(R_2')$.
By Lemma \ref{lem-s-hex-arc}, for any two vertices $l_1$, $l_2$ of $\cald(R_1)$, they form an edge of $\cald(R_1)$ if and only if there exists a hexagon in $\calcp_s(S)$ satisfying the assumption in Proposition \ref{prop-s-hex} and containing the HBPs $\{ \beta, \alpha \}$, $\{ \beta, \gamma_1\}$ and $\{ \beta, \gamma_2\}$, where $\gamma_1$ and $\gamma_2$ are HBCs in $R_1$ defined by the arcs $l_1$ and $l_2$, respectively.
The same property holds for $R_1'$, $\Psi(\alpha)$ and $\Psi(\beta)$ in place of $R_1$, $\alpha$ and $\beta$, respectively.
The map $\psi$ induces an injective simplicial map from $\cald(R_1)$ into $\cald(R_1')$, which is surjective by Proposition \ref{prop-d-inj}.
Similarly, $\psi$ induces a simplicial isomorphism from $\cald(R_2)$ onto $\cald(R_2')$.
The first equality in the claim is thus proved.
If $b$ is an HBP in $S$ with $i(\alpha, b)=0$ and $\alpha \not\in b$, then $b$ is a 2-HBP in $S$.
Since $\Psi(\alpha)$ separates the two components of $\partial S$, we have $\Psi(\alpha)\not\in \psi(b)$.
The second equality in the claim thus follows from the first equality.
\end{proof}

We now prove that $\Psi \colon V_s(S)\to V_s(S)$ defines a superinjective map from $\calc_s(S)$ into itself.
We first show that $\Psi$ defines a simplicial map from $\calc_s(S)$ into itself.
Let $\alpha$ and $\beta$ be separating curves in $S$ with $\alpha \neq \beta$ and $i(\alpha, \beta)=0$.
If $\alpha$ is an HBC, then there exists a curve $\beta'$ in $S$ such that $\{ \beta, \beta'\}$ is an HBP disjoint from $\alpha$.
Since $\psi(\alpha)$ and $\psi(\{ \beta, \beta'\})$ are disjoint, we have $i(\Psi(\alpha), \Psi(\beta))=0$.
If neither $\alpha$ nor $\beta$ is an HBC, then $\{ \alpha, \beta \}$ is an HBP because the genus of $S$ is equal to 2.
We thus obtain $i(\Psi(\alpha), \Psi(\beta))=0$.
It follows that $\Psi$ defines a simplicial map from $\calc_s(S)$ into itself.

We next pick two separating curves $\alpha$, $\beta$ in $S$ with $i(\Psi(\alpha), \Psi(\beta))=0$, and show $i(\alpha, \beta)=0$.
If both $\alpha$ and $\beta$ are HBCs in $S$, then we have $\Psi(\alpha)=\psi(\alpha)$ and $\Psi(\beta)=\psi(\beta)$, and thus $i(\alpha, \beta)=0$ by superinjectivity of $\psi$.

Suppose that $\alpha$ is not an HBC and that $\beta$ is an HBC.
Since $\Psi(\beta)$ is an HBC and is disjoint from $\Psi(\alpha)$, the curve $\Psi(\alpha)$ does not separate the two components of $\partial S$. 
It follows that $\alpha$ does not separate the two components of $\partial S$.
Choose a 2-HBP $a'$ in $S$ with $\Psi(\alpha)\in a'$ and $i(a', \Psi(\beta))=0$.
Claim \ref{claim-g2} implies that there exists an HBP $a$ in $S$ with $\alpha \in a$ and $\psi(a)=a'$.
Since we have $i(\psi(a), \psi(\beta))=i(a', \Psi(\beta))=0$, superinjectivity of $\psi$ implies $i(a, \beta)=0$ and thus $i(\alpha, \beta)=0$.

We finally suppose that neither $\alpha$ nor $\beta$ is an HBC.
Neither $\Psi(\alpha)$ nor $\Psi(\beta)$ is an HBC.
If $\Psi(\alpha)=\Psi(\beta)$, then Claim \ref{claim-g2} and injectivity of $\psi$ imply the equality
\[\{\, b\in U_s(S)\mid i(b, \alpha)=0\,\}=\{\, c\in U_s(S)\mid i(c, \beta)=0\,\}.\]
We thus have the equality $\alpha =\beta$ and particularly $i(\alpha, \beta)=0$.
If $\Psi(\alpha)\neq \Psi(\beta)$, then $\{ \Psi(\alpha), \Psi(\beta)\}$ is an HBP in $S$ because the genus of $S$ is equal to 2.
By Claim \ref{claim-g2}, there exist curves $\alpha_1$, $\beta_1$ in $S$ such that the pairs $\{ \alpha, \alpha_1\}$ and $\{ \beta, \beta_1\}$ are HBPs in $S$ and we have the equality $\psi(\{ \alpha, \alpha_1\})=\psi(\{ \beta, \beta_1\})=\{ \Psi(\alpha), \Psi(\beta)\}$.
Injectivity of $\psi$ implies the equality $\{ \alpha, \alpha_1\}=\{ \beta, \beta_1\}$ and particularly $i(\alpha, \beta)=0$.

We thus proved that $\Psi$ defines a superinjective map from $\calc_s(S)$ into itself, which is surjective by \cite[Theorem 1.1]{kida-cohop}.
It follows from Claim \ref{claim-g2} that $\psi$ is surjective.
\end{proof}


\subsection{The case of $p\geq 3$}

We first prove the following:

\begin{prop}\label{prop-p3-conn}
Let $S=S_{g, p}$ be a surface with $g\geq 2$ and $p\geq 2$. Then the following assertions hold:
\begin{enumerate}
\item The full subcomplex of $\calcp_n(S)$ spanned by all vertices corresponding to 2-HBCs or $p$-HBPs is connected. 
\item If $p\geq 3$, then the full subcomplex of $\calcp_s(S)$ spanned by all vertices corresponding to 2-HBCs or $p$-HBPs is connected.
\end{enumerate}
\end{prop}

\begin{proof}
We follow the idea in \cite[Lemma 2.1]{putman-conn} to prove connectivity of simplicial complexes on which $\pmod(S)$ acts, as in the proof of Lemma \ref{lem-d-y} (ii).
It is known that $\pmod(S)$ is generated by Dehn twists about the curves drawn in Figure \ref{fig-hbchbp} (a) (see \cite{gervais}).
\begin{figure}
\begin{center}
\includegraphics[width=12cm]{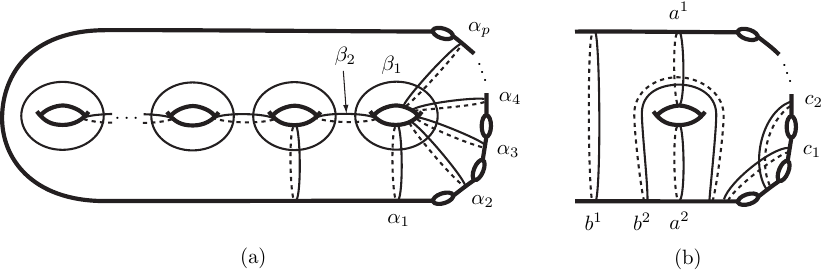}
\caption{(a) $\pmod(S)$ is generated by Dehn twists about these curves. (b) $\{ a^1, a^2\}$ is a non-separating $p$-HBP, and $\{ b^1, b^2\}$ is a separating $p$-HBP.}\label{fig-hbchbp}
\end{center}
\end{figure}
Let $a=\{ a^1, a^2\}$ denote the non-separating $p$-HBP in Figure \ref{fig-hbchbp} (b).
For any 2-HBC $\alpha$ in $S$, there exists $h\in \pmod(S)$ such that $\alpha$ is disjoint from $ha$.
To prove assertion (i), it thus suffices to show that for any curve $\gamma$ in Figure \ref{fig-hbchbp} (a), $t_{\gamma}a$ and $a$ can be connected by a path in $\calcp_n(S)$ consisting of vertices corresponding to $p$-HBPs or 2-HBCs.
This is true because $t_{\beta_1}a$ and $a$ can be connected via the 2-HBC $c_1$ in Figure \ref{fig-hbchbp} (b).
Assertion (i) follows.

Let $b=\{ b^1, b^2\}$ denote the separating $p$-HBP in Figure \ref{fig-hbchbp} (b).
For any 2-HBC $\alpha$ in $S$, there exists $h\in \pmod(S)$ such that $\alpha$ is disjoint from $hb$.
For any separating $p$-HBP $b'$ in $S$, there exists $h'\in \pmod(S)$ such that $b'$ is disjoint from $h'c_1$.
Since $b$ is disjoint from $c_1$, we have the path $b'$, $h'c_1$, $h'b$ in $\calcp_s(S)$.
To prove assertion (ii), it thus suffices to show that for any curve $\gamma$ in Figure \ref{fig-hbchbp} (a), $t_{\gamma}b$ and $b$ can be connected by a path in $\calcp_s(S)$ consisting of vertices corresponding to $p$-HBPs or 2-HBCs.
This is true because $t_{\alpha_2}b$ and $b$ can be connected via the 2-HBC $c_2$ in Figure \ref{fig-hbchbp} (b); for any $j=3,\ldots, p$, $t_{\alpha_j}b$ and $b$ can be connected via $c_1$; and $t_{\beta_2}b$ and $b$ can be connected via $c_1$.
\end{proof}

\begin{thm}\label{thm-surj-p3}
Let $S=S_{g, p}$ be a surface with $g\geq 2$ and $p\geq 3$.
Let $\phi \colon \calcp_n(S)\to \calcp_n(S)$ and $\psi \colon \calcp_s(S)\to \calcp_s(S)$ be superinjective maps.
Then $\phi$ and $\psi$ are surjective.
\end{thm}

\begin{proof}
In Lemmas \ref{lem-phi-nskhbp} and \ref{lem-phi-jhbc}, we have already shown that for each $j$ and each $k$, the map $\phi$ preserves $j$-HBPs and $k$-HBCs, respectively.
We prove surjectivity of $\phi$ by induction on $p$.
For any 2-HBC $\alpha$ in $S$, $\phi$ induces a map from the link of $\alpha$ in $\calcp_n(S)$ into that of $\phi(\alpha)$.
This induced map can be identified with a superinjective map from $\calcp_n(S_{g, p-1})$ into itself, which is surjective by the hypothesis of the induction.
The image of $\phi$ thus contains all vertices of the link of $\phi(\alpha)$.
For any non-separating $p$-HBP $b$ in $S$, $\phi$ induces a map from the link of $b$ in $\calcp_n(S)$ into that of $\phi(b)$.
For any curve $\beta$ in $b$, the restriction of this induced map to the set of HBCs and HBPs containing $\beta$ can be identified with a superinjective map from $\calc(S_{0, p+2})$ into itself.
Since the latter map is surjective by \cite[Theorem 2]{bm-ar}, the image of $\phi$ contains all vertices of the link of $\phi(b)$.
By Proposition \ref{prop-p3-conn} (i), $\phi$ is surjective.

Similarly, we can prove surjectivity of $\psi$ by using Proposition \ref{prop-p3-conn} (ii).
\end{proof}

\begin{cor}
Let $S=S_{g, p}$ be a surface with $g\geq 2$ and $p\geq 2$.
Then any superinjective map from $\calcp(S)$ into itself is surjective.
\end{cor}

\begin{proof}
Let $\phi \colon \calcp(S)\to \calcp(S)$ be a superinjective map.
By Lemma \ref{lem-bn-b}, we have the inclusion $\phi(\calcp_n(S))\subset \calcp_n(S)$.
By Theorems \ref{thm-bn-surj} and \ref{thm-surj-p3}, we have the equality $\phi(\calcp_n(S))=\calcp_n(S)$.
Injectivity of $\phi$ implies the inclusion $\phi(\calcp_s(S))\subset \calcp_s(S)$.
By Theorems \ref{thm-bs-surj} and \ref{thm-surj-p3}, we have the equality $\phi(\calcp_s(S))=\calcp_s(S)$.
\end{proof}


\subsection{Superinjective maps from $\calcp_n(S)$ into $\calcp(S)$}

Let $S=S_{g, p}$ be a surface with $g\geq 2$ and $p\geq 2$.
Let $\phi$ be an automorphism of $\calcp_n(S)$.
We denote by $\calc_{nc}(S)$ the full subcomplex of $\calc(S)$ spanned by all vertices corresponding to either a non-separating curve in $S$ or an HBC in $S$.
Using Lemmas \ref{lem-n-well-p3} and \ref{lem-n-well} and following a part of the proof of \cite[Theorem 6.1]{kida-yama}, we obtain a simplicial automorphism $\Phi$ of $\calc_{nc}(S)$ with $\phi(\{ \alpha, \beta\})=\{ \Phi(\alpha), \Phi(\beta)\}$ for any non-separating HBP $\{ \alpha, \beta \}$ in $S$ and with $\phi(\gamma)=\Phi(\gamma)$ for any HBC $\gamma$ in $S$.
As asserted in \cite[Remark in p.102]{irmak-ns}, the map $\Phi$ extends to a simplicial automorphism of $\calc(S)$, and is thus induced by an element of $\mod^*(S)$ by \cite[Theorem 1]{iva-aut}.
We proved the following:

\begin{thm}\label{thm-aut-bn}
Let $S=S_{g, p}$ be a surface with $g\geq 2$ and $p\geq 2$.
Then any automorphism of $\calcp_n(S)$ is induced by an element of $\mod^*(S)$.
\end{thm}

Combining Lemma \ref{lem-bn-b} and Theorems \ref{thm-bn-surj} and \ref{thm-surj-p3}, we obtain the following:

\begin{cor}\label{cor-bn}
Let $S$ be the surface in Theorem \ref{thm-aut-bn}.
Then any superinjective map from $\calcp_n(S)$ into $\calcp(S)$ is induced by an element of $\mod^*(S)$.
\end{cor}


\end{document}